\documentclass[12pt]{amsart}
\usepackage[hmargin=1in,vmargin=1.25in]{geometry}
\usepackage{float}
\usepackage[utf8]{inputenc}
\usepackage{amsmath}
\usepackage{amsthm}
\usepackage{amssymb}
\usepackage{amsfonts}
\usepackage{enumerate}
\usepackage{tikz}
\usepackage{tkz-euclide}
\usepackage{ytableau}
\usepackage{quiver}
\usepackage{longtable}
\usepackage{mathtools}

\theoremstyle{plain}
\newtheorem{theorem}{Theorem}[section]
\newtheorem{lemma}[theorem]{Lemma}
\newtheorem{corollary}[theorem]{Corollary}
\newtheorem{proposition}[theorem]{Proposition}
\newtheorem{conjecture}[theorem]{Conjecture}
\newtheorem{problem}[theorem]{Problem}

\theoremstyle{definition}
\newtheorem{defin}[theorem]{Definition}

\newtheorem{remark}[theorem]{Remark}
\newtheorem{exmp}[theorem]{Example}

\def\bbC{{{\mathbb C}}}
\def\bbR{{{\mathbb R}}}

\def\bbZ{{{\mathbb Z}}}

\def\bbQ{{{\mathbb Q}}}

\def\bx{{{\mathbf x}}}
\def\by{{{\mathbf y}}}
\def\bz{{{\mathbf z}}}

\def\bv{{{\mathbf v}}}
\def\bw{{{\mathbf {w}}}}
\def\bu{{{\mathbf {u}}}}

\def\des{{{\rm des}}}
\def\cyc{{{\rm cyc}}}
\def\bi{{{\mathbf {i}}}}

\def\bh{{{\mathbf {h}}}}

\def\ba{{{\mathbf {a}}}}
\def\bb{{{\mathbf {b}}}}
\def\bc{{{\mathbf {c}}}}
\def\bd{{{\mathbf {d}}}}
\def\br{{{\mathbf {r}}}}
\def\bs{{{\mathbf {s}}}}
\def\bt{{{\mathbf {t}}}}
\def\bP{{{\mathbf {P}}}}
\def\bN{{{\mathbf {N}}}}
\def\bCN{{{\mathbf {CN}}}}
\def\wt{{{\rm wt}}}

\def\skel{{{\rm skel}}}

\ytableausetup{boxsize=1.5em}

\newcommand{\inc}{\text{\rm inc}}

\newcommand{\SSYT}{\text{\rm SSYT}}
\newcommand{\SYT}{\text{\rm SYT}}
\newcommand{\T}{\text{\rm T}}
\newcommand{\ST}{\text{\rm ST}}

\newcommand{\CSSYT}{\text{\rm CSSYT}}

\newcommand{\threeone}{\ensuremath{(\mathbf{3}+\mathbf{1})}}
\newcommand{\twotwo}{\ensuremath{(\mathbf{2} + \mathbf{2})}}
\newcommand{\sgn}{\text{\rm sgn}}

\newcommand{\sort}{\text{\rm sort}}

\newcommand{\stat}{\text{\rm stat}}

\newcommand{\lpk}{\text{\rm lpk}}

\newcommand{\Sp}{\text{\rm Sp}}
\newcommand{\Tr}{\text{\rm Tr}}
\newcommand{\imm}{\text{\rm Imm}}
\newcommand{\A}{\text{\rm A}}
\newcommand{\SA}{\text{\rm SA}}
\newcommand{\chain}{\text{\rm chain}}

\numberwithin{equation}{section}

\title[]{Schur-Positve Symmetric Functions from Products: Log-Majorization, Total-Positivity, and PSD Matrices}

   \author{Isaiah Siegl}

   \address
   {Dept.\ of Mathematics\\
    University of Washington\\
    Seattle, WA}
   \date{\today}
   \email{isaiahsiegl@gmail.com}

\begin{document}

\begin{abstract}
For a polynomial $F(t) = 1 + a_1 t + a_2 t^2 +\dots + a_d t^d$ with coefficients in a $\mathbb{Q}$-algebra $\mathcal{R}$, we consider Schur-positivity of the symmetric function $\Omega_F(\bx) = \prod_{i \geq 1} F(x_i)$. 
Using work of Lindstr\"om and Gessel--Viennot, we give sufficient conditions for Schur-positivity in terms of directed networks. In the case when $F(t)$ is the multivariate independence polynomial of the  incomparability graph $\inc(P)$ of a natural unit interval order, we apply these ideas to give a new proof of Schur-positivity of the chromatic symmetric function $X_{\inc(P)}(\bx)$. Additionally, we prove Schur-positivity of certain sums of chromatic symmetric functions, refining recent work of Amdeberhan--Shareshian--Stanley.
In the case when $\mathcal{R} = \mathbb{Q}[z_1,z_2,\dots,z_n]$, we prove necessary and sufficient conditions for Schur-positivity in terms of symmetric function evaluations. In particular, we show for $\mathbf{u}, \mathbf{v} \in \mathbb{R}_{\geq 0}^n$, $s_\lambda(\mathbf{u}) \geq s_\lambda(\mathbf{v})$ for all partitions $\lambda$ if and only if $m_\lambda(\mathbf{u}) \geq m_\lambda(\mathbf{v})$ for all partitions $\lambda$, and these inequalities are equivalent to $\mathbf{u}$ log-majorizing $\mathbf{v}$. 
Using representation theory of $GL(\mathbb{C}^n)$, we prove Schur-positivity of $\Omega_F(\bx)$ when $F(t) = \det(I + t\sum_{i = 1}^m z_i \bv_i \bv_i^*)$ for a collection of vectors $\bv_1,\bv_2,\dots, \bv_m \in \mathbb{C}^n$. 
As an application for this result, we introduce the \emph{Kirchhoff symmetric function} $K_G(\bx)$ of a loopless graph $G$ and show that it it always Schur-positive.
\end{abstract}%
\maketitle
\section{Introduction}
Many fundamental problems in algebraic combinatorics concern the Schur-positivity of families of symmetric functions. Such problems often have deep connections to geometry and representation theory and reveal interesting combinatorial structure. 
A natural way to obtain symmetric functions is by taking products of polynomials.
For a polynomial $F(t) = 1 + a_1 t + a_2 t^2 + \dots + a_d t^d$, 
define $\Omega_F(\bx)$ by
\begin{equation}
\label{eq Omega F def}
\Omega_F(\bx) := \prod_{i \geq 1} F(x_i) = (1 + a_1 x_1 + \dots + a_d x_i^d)(1 + a_1 x_2 + \dots + a_d x_2^d) \cdots
\end{equation}
 
We say $\Omega_F(\bx)$ is Schur-positive if each homogeneous component is Schur-positive.
Motivated by positivity problems for chromatic symmetric functions and building on work of Aissen--Schoenberg--Whitney \cite{ASW}, Stanley showed that Schur-positivity of $\Omega_F(\bx)$ is closely related to real-rootedness of $F(t)$ and total-nonnegativity of Toeplitz matrices.
\begin{theorem}
\label{thm stan poly s pos}
\cite{StanleyChromatic1998}
Let $F(t) = 1 + a_1 t + a_2 t^2 + \dots + a_dt^d \in \mathbb{R}[t]$. The following are equivalent.
\begin{enumerate}
\item All roots of $F(t)$ are negative real numbers.
\item $\Omega_F(\bx)$ is Schur-positive.
\item $\Omega_F(\bx)$ is $e$-positive.
\item The Toeplitz matrix $\left[a_{j-i}\right]_{i,j=1}^{\infty}$ is totally-nonnegative, where $a_i$ is taken to be 0 when $i < 0$.
\end{enumerate}
\end{theorem}

The purpose of this paper is to study Schur-positivity and $e$-positivity of $\Omega_F(\bx)$ when the coefficients of $F(t)$ are elements of a not-necessarily-commutative $\bbQ$-algebra $\mathcal{R}$.
For $\Omega_F(\bx)$ to be a well-defined power series, we require $F(0) = 1$, and for $\Omega_F(\bx)$ to be a symmetric function, we require that the coefficients of $F(t)$ to commute with each other in $\mathcal{R}$. 
Write $\Lambda = \bbQ[x_1,x_2,\dots]^{S_\infty}$ for the ring of symmetric functions of bounded degree.
A set $\mathcal{R}_+ \subset \mathcal{R}$ is a \emph{positive cone} if it is closed under multiplication and non-negative $\bbQ$-linear combinations.
We say a $\mathcal{R}$-coefficient symmetric function $f \in \mathcal{R} \otimes \Lambda$ is Schur-positive (resp. $e$-positive) if $f$ expands in the basis of Schur functions (resp. elementary symmetric functions) with coefficients in $\mathcal{R}_+$.
The equivalence between Schur-positivity of $\Omega_F(\bx)$ and the total-nonnegativity of the associated Toeplitz matrix extends to this generalized setting. However, Schur-positivity of $\Omega_F(\bx)$ does not imply $e$-positivity for general $\mathcal{R}$. 

%
A useful tool for studying the positivity of $\Omega_F(\bx)$ is the following \emph{dual Cauchy identity}
\begin{align}
\Omega(\bx, \by) &:= \prod_{i, j} (1 + x_i y_j) = \prod_{i} \sum_{k \geq 0} e_k(\by) x_i^k \\
&= \sum_{\lambda} e_\lambda(\bx) m_{\lambda}(\by) = \sum_{\lambda} s_\lambda(\bx) s_{\lambda'}(\by) \\
&= \sum_\lambda (-1)^{|\lambda| - \ell(\lambda)} z_\lambda^{-1} p_\lambda(\bx) p_\lambda(\by) = \sum_{\lambda} m_\lambda(\bx) e_{\lambda}(\by).
\end{align}
In the above identity, $\lambda'$ denotes the conjugate of the partition $\lambda$, and $z_\lambda = \prod_j m_j!\, j^{m_j}$, where $m_j$ is the number of parts of size $j$ in $\lambda$. For $F(t) = 1 + e_1^F t + \dots + e_d^F t^d \in \mathcal{R}[t]$, define the \emph{elementary $F$-analogue} $e_\lambda^F \in \mathcal{R}$ by $e_\lambda^F = e_{\lambda_1}^F e_{\lambda_2}^F \dots e_{\lambda_\ell}^F$. 
For a symmetric function $f(\bx) = \sum_\lambda a_\lambda e_\lambda(\bx)$, define the \emph{$F$-analogue} $f^F$ of $f$ by replacing each elementary symmetric function with its $F$-analogue so that $f^F = \sum_{\lambda} a_\lambda e_\lambda^F$. 
Thus, for any symmetric function identity, we have a corresponding identity for $F$-analogues. For example, the Jacobi--Trudi identity $s_\lambda(\bx) = \det(e_{\lambda_i' -i+ j}(\bx))$ implies $s_\lambda^F = \det(e_{\lambda_i' - i + j}^F)$. Corresponding to the dual Cauchy identity, we have
\begin{equation}
\label{eq F cauchy id}
\Omega_F(\bx) = \prod_{i} F(x_i) = \sum_{\lambda} e_\lambda^F m_\lambda(\bx) = \sum_{\lambda} s_\lambda^F s_{\lambda'}(\bx) = \sum_{\lambda} m_\lambda^F e_\lambda(\bx).
\end{equation}
As an immediate consequence of \eqref{eq F cauchy id}, we have that $\Omega_F(\bx)$ is monomial-positive if and only if each $e_i^F$ is in $\mathcal{R}_+$. Furthermore, $\Omega_F(\bx)$ is Schur-positive if and only if $s_\lambda^F \in \mathcal{R}_+$ for all partitions $\lambda$, and $\Omega_F(\bx)$ is $e$-positive if and only if $m_\lambda^F \in \mathcal{R}_+$ for all partitions $\lambda$. 

For partitions $\mu \subset \lambda$ with $m$ columns, the skew Schur function $s_{\lambda/\mu}(\bx)$ is the $(I, J)$-minor of the Toeplitz matrix $[e_{j-i}(\bx)]_{i,j = 1}^\infty$, where
$I = \{\mu_m' +1 , \mu_{m - 1}' + 2, \dots, \mu_1' + m\}$ and $J = \{\lambda_m' + 1, \lambda_{m - 1}' + 2, \dots, \lambda_1' + m\}$.
Applying the definition of $F$-analogues, this gives $s_{\lambda/\mu}^F = \Delta_{I, J}[e_{j-i}^F]_{i,j = 1}^\infty$. Say a matrix $A$ with entries in $\mathcal{R}$ is totally-nonnegative with respect to $\mathcal{R}_+$ if each minor of $A$ is in $\mathcal{R}_+$. The following proposition gives an equivalence between Schur-positivity of $\Omega_F(\bx)$ and total-nonnegativity of $[e_{j-i}^F]_{i,j = 1}^\infty$. The proof follows directly from the definitions.
\begin{proposition}
\label{prop TP s pos}
Let $F(t) = 1 + e_1^F t + \dots + e_d^F t^d \in \mathcal{R}[t]$. Then $\Omega_F(\bx)$ is Schur-positive with respect to a positive cone $\mathcal{R}_+$ if and only if the matrix $\left[e_{j-i}^F\right]_{i,j = 1}^\infty$ is totally-nonnegative with respect to $\mathcal{R}_+$.
\end{proposition}

Given a factorization $F(t) = \prod_{i = 1}^d (1 + u_i t)$ with commuting factors in $\mathcal{R}[t]$, the connection between $\Omega_F(\bx)$ and $\Omega(\bx, \by)$ becomes more explicit.
As the coefficients of a polynomial are related to its by the elementary symmetric functions, we have 
\begin{equation}
F(t) = \sum_{k = 0}^d e_k^F t^k = \sum_{k = 0}^d e_k(\bu) t^d = \prod_{i = 1}^d ( 1 + u_i t).
\end{equation}
Therefore, $e_k^F = e_k(\bu)$ for all $k$, and for any symmetric function $f$, we have $f^F = f(\bu)$. As a consequence, $\Omega_F(\bx) = \Omega(\bu, \bx)$. Therefore, if each $u_i$ is in $\mathcal{R}_+$, then $\Omega_F(\bx)$ is $e$-positive.

We are particularly interested in the case when $\mathcal{R}$ is the commutative polynomial ring $\bbQ[z_1,z_2,\dots, z_n]$ and $\mathcal{R}_+ = \bbQ_{\geq 0}[z_1,z_2,\dots, z_n].$
In this case, there is a useful interplay between evaluations of $F$-analogues $f^F(\bz)$ and evaluations of the corresponding symmetric function $f(\bx)$.
For $F(\bz, t) \in \mathcal{R}[t]$, write $\Omega_F(\bz, \bx) := \Omega_F(\bx)$. 
The following proposition follows immediately from Theorem~\ref{thm stan poly s pos} and the fact that evaluating polynomials in $\bbQ_{\geq 0}[\bz]$ at nonnegative real vectors gives nonnegative real numbers. 
\begin{proposition}
\label{prop bbQ Z eval roots}
Let $F(\bz, t) \in \bbQ[z_1,z_2,\dots,z_n,t]$. If $\Omega_F(\bz, \bx)$ is Schur-positive, then all roots of $F(\br, t)$ are negative real numbers for all vectors $\br \in \bbR_{\geq 0}^n$.
\end{proposition}

Suppose $\Omega_F(\bz, \bx)$ is Schur-positive.
As a consequence of Proposition~\ref{prop bbQ Z eval roots}, for any $\br \in \bbR_{\geq 0}^n$, we have $F(\br, t) = \prod_{i = 1}^d (1 + u_i t)$  for some $\bu \in \bbR_{\geq 0}^d$. Observe that for each $k$, we have $e_k^F(\br) = e_k(\bu)$. Hence, for any symmetric function $f(\bx)$, we have $f^F(\br) = f(\bu)$. 
Let $\br, \bs \in \bbR_{\geq 0}^n$ such that $r_i \geq s_i$ for all $1 \leq i \leq n$,  and let $F(\bs, t) = \prod_{i=1}^d (1 + v_i t)$ for some $\bv \in \bbR_{\geq 0}^d$.
Then for any  polynomial $p(\bz) \in \bbQ_{\geq 0}[\bz]$, we have $p(\br) \geq p(\bs)$ as real numbers. As $\Omega_F(\bz, \bx)$ is Schur-positive, then for each partition $\lambda$, $s_\lambda^F(\bz) \in \bbQ_{\geq 0}[\bz]$ and therefore
\begin{equation}
\label{eq sF ineq}
s_\lambda(\bu) = s_\lambda^F(\br) \geq s_\lambda^F(\bs) = s_\lambda(\bv).
\end{equation}
Likewise, if $\Omega_F(\bz, \bx)$ is $e$-positive, then for all partitions $\lambda$, we have
\begin{equation}
\label{eq mF ineq}
m_\lambda(\bu) = m_\lambda^F(\br) \geq m_\lambda^F(\bs) = m_\lambda(\bv). 
\end{equation}
These inequalities suggest an interesting relationship between Schur- and $e$-positivity of $\Omega_F(\bz, \bx)$ and symmetric function evaluations.
The following theorem characterizes the inequalities \eqref{eq sF ineq} and \eqref{eq mF ineq} in terms of \emph{log-majorization}.
This result mirrors, but is distinct from, the relationship between majorization and \emph{symmetric means} studied by Cuttler--Greene--Skandera \cite{CGSsymMeans}, Sra \cite{Sra}, and Khare--Tao \cite{KhareTao}.
Additionally, this result characterizes when the difference of classical dual Cauchy products $\Omega(\bu, \bx) - \Omega(\bv, \bx)$ is Schur-positive and $e$-positive.
\begin{theorem}
\label{thm sym eval ineq}
Let $\bu, \bv \in \mathbb{R}^n_{\geq 0}$.
The following are equivalent.
\begin{enumerate}
\item $\bu \succeq_{\log} \bv$,
\item $m_\lambda(\bu) \geq m_\lambda(\bv)$ for all partitions $\lambda$,
\item $s_\lambda(\bu) \geq s_\lambda(\bv)$ for all partitions $\lambda$,
\item $\Omega(\bu, \bx) - \Omega(\bv, \bx)$ is Schur-positive,
\item $\Omega(\bu, \bx) - \Omega(\bv, \bx)$ is $e$-positive.
\end{enumerate}
\end{theorem}

Our study of these symmetric functions $\Omega_F(\bz, \bx)$ is motivated by positivity problems for chromatic symmetric functions. For a graph $G$, the chromatic symmetric function $X_G(\bx)$ counts proper colorings of $G$. 
In 1993, Stanley--Stembridge \cite{StanleyStembridge} conjectured that $X_G(\bx)$ is $e$-positive when $G$ is the incomparability graph $\inc(P)$ of a \threeone -free poset $P$. 
In this case, Gasharov \cite{Gasharov} gave a combinatorial interpretation of the Schur coefficients of $X_{\inc(P)}(\bx)$ in terms of \emph{$P$-tableaux}. 
Guay-Paquet \cite{GPchromatic} reduced the Stanley--Stembridge conjecture to the case of \emph{natural unit interval orders} $P_\bh$, which are parameterized by \emph{Hessenberg functions} $\bh \in \mathbb{H}_n$.
Shareshian--Wachs \cite{SWchromatic} conjectured and Brosnan--Chow \cite{BrosnanChow} and Guay-Paquet \cite{guaypaquetCohom} proved that $\omega X_{\inc(P_\bh)}(\bx)$ is the Frobenius characteristic of the $S_n$ representation on the equivariant cohomology ring of regular semisimple Hessenberg varieties based on work of Tymoczko \cite{Tymoczko}.
In 1998, Stanley showed that Schur-positivity and $e$-positivity of chromatic symmetric functions are closely related to Schur and $e$-positivity of $\Omega^G(\bz, \bx) := \Omega_{I_G}(\bz, \bx) = \prod_{i \geq 1} I_G(\bz, x_i)$, where $I_G(\bz, t) = \sum_{k \geq 0} e_k^G(\bz) t^k$ is the \emph{multivariate independence polynomial of $G$}. Specifically, the Stanley--Stembridge conjecture is equivalent to $e$-positivity of $\Omega^G(\bz, \bx)$ for all incomparability graphs $G$ of \threeone -free posets. In a recent landmark result, Tatsuyuki Hikita \cite{HikitaChromatic} showed that $X_{\inc(P_\bh)}(\bx)$ is $e$-positive, resolving the Stanley--Stembridge conjecture.
We build on the known Schur-positivity and $e$-positivity results for natural unit interval orders in the results below.
We recall further relevant background for chromatic symmetric functions in Section~\ref{sec chromaticbackground}. 

Using Proposition~\ref{prop TP s pos}, we give a new approach to proving Schur-positivity of chromatic symmetric functions through total-positivity. In particular, for a Hessenberg function $\bh \in \mathbb{H}_n$, we construct a directed network $\Gamma_\bh$ with path-weight matrix $W_{\Gamma_\bh} = \left[e_{j-i}^{\inc(P_\bh)}(\bz)\right]_{i,j = 1}^\infty$. From Lindstr\"om's lemma, we know that $W_{\Gamma_\bh}$ is totally-nonnegative and therefore $\Omega^{\inc(P)}(\bz, \bx)$ is Schur-positive.
Theorem~\ref{thm s pos XincP Gamma h} shows how the coefficient of $s_\lambda(\bx)$ in $X_{\inc(P_\bh)}(\bx)$ can be recovered from non-intersecting path families in $\Omega_\bh$. 
In Theorem~\ref{thm s pos matching sum}, we use a modification of $\Gamma_\bh$ to prove Schur-positivity of sums of not-necessarily-Schur-positive chromatic symmetric functions.
Theorem~\ref{thm s pos matching sum} is closely related to a recent result of Amdeberhan--Shareshian--Stanley \cite{AmdShaSta}.
\begin{theorem}
\label{thm s pos XincP Gamma h}
Let $\bh \in \mathbb{H}_n$ be a Hessenberg function, and let $X_{\inc(P_\bh)}(\bx) = \sum_{\lambda \vdash n} a_\lambda s_\lambda(\bx)$. Then for a partition $\lambda = (\lambda_1, \dots, \lambda_\ell)$, the Schur-coefficient $a_{\lambda'}$ is the number of non-intersecting path families in $\Gamma_\bh$ from sources $(s_1,s_2,\dots, s_\ell)$ to sinks $(t_{\lambda_\ell}, t_{\lambda_{\ell-1}+1},\dots, t_{\lambda_1 + \ell - 1})$
In particular, there is a natural bijection between non-intersecting path families in $\Gamma_\bh$ and Gasharov's $P_\bh$-tableaux.
\end{theorem}
\begin{theorem}
\label{thm s pos matching sum}
Let $\bh \in \mathbb{H}_n$, and let $\bd \in \bbZ_{\geq 0}^{n \times r}$. The symmetric function
\begin{equation}
\sum_{M} \frac{X_{\inc(P_M)}(\bx) }{\prod_{\bc \in \chain_r(P_\bh)} m_{\bc}(M)!}
\end{equation}
is Schur-positive, where the sum is over all multisets $M$ of chains of length $r$ in $P_\bh$ such that $\deg(M) = \bd$, $P_M$ is the interval order associated to $M$, and $m_\bc(M)$ is the multiplicity of a chain $\bc$ in $M$.
\end{theorem}

To connect symmetric functions $\Omega_F(\bz, \bx)$ to the third part of the title, consider the class of real-rooted polynomials $F(t) = \det(I + tA)$ for positive semidefinite (PSD) matrices $A$.
By the spectral theorem, if $A$ is a positive semidefinite matrix, then all roots of the polynomial $\det(I + t A)$ are negative real numbers. For any collection of vectors $\bv_1,\bv_2,\dots, \bv_m \in \mathbb{C}^{m}$, the matrix $A(\bz) = \sum_{i = 1}^m z_i \bv_i\bv_i^*$ is positive semidefinite for all $\bz \in \bbR_{\geq 0}^m$. 
In the following theorem, we show that $\Omega_F(\bz, \bx)$ is Schur-positive for $F(\bz, t) = \det(I + t A(\bz))$.

\begin{theorem}
\label{thm herm s pos}
For $\bv_1, \bv_2, \dots,\bv_m \in \mathbb{C}^n$, let $A(\bz) = \sum_{i = 1}^m z_i \bv_i \bv_i^*$, where $\bv^*$ is the conjugate transpose of $\bv$. Let $F(\bz, t) = \det(I + t A(\bz))$. Then $\Omega_F(\bz, \bx)$ is Schur-positive.
\end{theorem}

The paper is organized as follows. In Section 2, we recall necessary background and prove Proposition~\ref{prop TP s pos}, generalizing work of Stanley \cite{StanleyChromatic1998}. In Section 3, we  
prove Theorem~\ref{thm sym eval ineq}. Using this result, we obtain log-majorization results for the roots of independence polynomials of incomparability graphs of \threeone -free posets. For $F(\bz, t) \in \mathbb{Q}[\bz, t]$, we give necessary conditions for Schur and $e$-positivity of $\Omega_F(\bz, \bx)$ in terms of alternating sums of symmetric function evaluations.
In Section 4, we use networks to study Schur-positivity of $\Omega_F(\bx)$. 
In particular, we prove Theorem~\ref{thm s pos XincP Gamma h} and Theorem~\ref{thm s pos matching sum}.
We state a refinement of a conjecture of Alexandersson--Dai \cite{AlexanderssonDai} on the map $T_k: \Lambda \to \Lambda$ sending $e_i(\bx)$ to $e_{ki}(\bx)$.
In Section 5, we prove Theorem~\ref{thm herm s pos}. To give an application to Theorem~\ref{thm herm s pos}, we introduce the \emph{Kirchhoff symmetric function} $K_G(\bx)$, which counts rooted forests of a loopless graph $G$, and show that it is always Schur-positive.
In Section 6, we give concluding remarks and state open problems.
As an example of symmetric functions $\Omega_F(\bx)$ with coefficients in a noncommutative ring, we consider the case when $\mathcal{R}$ is the symmetric group algebra $\mathbb{Q}S_n$. Using the factorization into Jucys--Murphy elements, we show that $\Omega_{C_n}(\bx)$ is $e$-positive when $C_n(t) = \sum_{\pi \in S_n} \pi t^{n - \cyc(\pi)}$, where $\cyc(\pi)$ is the number of cycles of $\pi$. For the $\mathbb{Q}S_n$-coefficient Eulerian polynomial $E_n(t) = \sum_{\pi \in S_n} \pi t^{\des(\pi)}$, we conjecture that $\Omega_{E_n}(\bx)$ is Schur-positive. 

\section{Preliminaries}
We assume familiarity with symmetric function theory and posets. Appropriate background can be found in \cite{ec2}. 

\begin{defin}
For a partition $\lambda = (\lambda_1, \lambda_2,...,\lambda_\ell)$, we identify $\lambda$ with the Ferrers diagram $\lambda = \{(i,j) \in \mathbb{Z}_{> 0}\mid j \leq \lambda_i \}$, drawn as boxes in the plane.
For partitions $\mu \subset \lambda$, define the \emph{skew shape} $\lambda/ \mu$ to be the set difference of their Ferrers diagrams $\lambda \setminus \mu$.
A filling $T$ of $\lambda/\mu$ with positive integers is a \emph{(skew) semistandard Young tableau} if $T_{i,j} \leq T_{i,j+1}$ and $T_{i,j} < T_{i+1, j}$ for all such pairs of cells in $\lambda/\mu$. Write $\SSYT(\lambda/\mu)$ for the set of semistandard Young tableaux of shape $\lambda/\mu$, and write $\SSYT(\lambda/\mu, n)$ for the set of semistandard Young tableaux with entries in $[n]$. Write $\SYT(\lambda/\mu)$ for the set of \emph{standard Young tableaux} of shape $\lambda/\mu$. A tableau $T \in \SSYT(\lambda)$ is a \emph{standard Young tableau} if each element of $\{1,2,...,|\lambda|\}$ appears exactly once in $T$. Write $\SYT(\lambda)$ for the set of standard Young tableaux of shape $\lambda$.
\end{defin}

The \emph{Weyl character formula} gives a nice expression for the number of semistandard Young tableau of shape $\lambda$ with entries in $[n]$. As a consequence, we have that $|\SSYT(m \lambda, n)|$ is a polynomial in  $m$, where $m \lambda = (m \lambda_1, m \lambda_2, \dots, m\lambda_\ell)$.

\begin{theorem}[Weyl character formula]
\cite[Theorem 6.3]{FultonHarris}
Let $\lambda$ be a partition, then
\begin{equation}
|\SSYT(\lambda, n)| = s_\lambda(1^n) = \prod_{1 \leq i < j \leq n} \frac{\lambda_i - \lambda_j + j - i}{j-i}.
\end{equation}
\end{theorem}

\begin{corollary}
\label{cor ssyt pol}
Let $\lambda$ be a partition, and let $m,n \in \mathbb{Z}_{\geq 0}$. Then
\begin{equation}
|\SSYT(m \lambda, n)| = s_{m \lambda}(1^n) = \prod_{1 \leq i < j \leq n} \frac{m(\lambda_i - \lambda_j) + j - i}{j-i}.
\end{equation}
Hence, $|\SSYT(m \lambda, n)|$ is a polynomial in $m$.
\end{corollary}

\subsection{Noncommutative Schur Functions and Sprout Symmetric Functions}

Symmetric functions of the form $\Omega_F(\bx) = \prod_{i \geq 1} F(x_i)$ for a power series $F(t) = \sum_{k \geq 0} e_k^F t^k$ have been studied in various settings.
The study of \emph{noncommutative Schur functions} was initiated by Fomin--Greene \cite{FG}. Noncommutative Schur functions can be interpreted as a special case of $s_\lambda^F(\bz)$, where $\mathcal{R}$ is a quotient of the noncommutative polynomial ring $\mathbb{Q}\langle z_1,z_2,\dots,z_n\rangle$ and $e_k^F \equiv \sum_{i_1 > i_2 > \dots > i_k} z_{i_1}z_{i_2}\cdots z_{i_k}$.
This approach was further developed by Blasiak--Fomin \cite{BF, BFcommute} and used to study positivity of a wide range of symmetric functions \cite{BLamLLT, BEPS, BLKronecker, Hwang, LamRibbon}.
More recently, Amdeberhan--Shareshian--Stanley \cite{AmdShaSta} have applied Theorem~\ref{thm stan poly s pos} to study positivity of various symmetric functions.
Amdeberhan--Shareshian--Stanley call $F(t)$ the \emph{seed} of the sequence of \emph{sprout symmetric functions} $(R_0, R_1, \dots)$ defined to be the homogeneous components of $\Omega_F(\bx)$.

Schur-positivity of $\Omega_F(\bx)$ is closely-related to \emph{P\'olya frequency} sequences and total nonnegativity of Toeplitz matrices. Recall that a matrix with real entries is defined to be \emph{totally nonnegative} if every minor is nonnegative. A sequence $(a_i)_{i \in \bbZ_{\geq 0}}$ is a \emph{P\'olya frequency sequence} if the Toeplitz matrix $\left[a_{j-i}\right]_{i,j = 1}^\infty$ is totally nonnegative. P\'olya frequency sequences were first studied in the 1950s by Schoenberg along with Aissen, Edrei, and Whitney \cite{ASW, Edrei}. P\'olya frequency sequences were later connected to the representation theory of the infinite symmetric group by Thoma \cite{Thoma}. The following theorem characterizing P\'olya frequency sequences was proved independently by Edrei \cite{Edrei} and Thoma \cite{Thoma}.

\begin{theorem}
\cite{Edrei, Thoma}
Let $F(t) = 1 + a_1 t + a_2 t^2 + \dots  \in \mathbb{R}[[t]].$
The Toeplitz matrix $\left[a_{j-i}\right]_{i,j = 0}^\infty$ is totally nonnegative if and only if $F(t)$ can be factored as
\begin{equation}
F(t) = \exp(\gamma t) \prod_{i \geq 1} \frac{1 + \beta_i t}{1 - \alpha_i t} 
\end{equation}
for real numbers $\gamma, \alpha_i, \beta_i \geq 0$ and $\sum_{i \geq 1} \alpha_i + \beta_i < \infty$.
\end{theorem}

Stanley connected P\'olya frequency sequences to Schur and $e$-positivity of $\Omega_F(\bx)$, proving Theorem~\ref{thm stan poly s pos} \cite{StanleyChromatic1998}. In particular, the coefficient of $s_\lambda(\bx)$ in $\Omega_F(\bx)$ is a minor of the Toeplitz matrix $\left[e_{j-i}^F\right]_{i,j=1}^\infty$ as in Proposition~\ref{prop TP s pos}.

\subsection{Chromatic Symmetric Functions}
\label{sec chromaticbackground}
For a graph $G = (V, E)$, a \emph{proper coloring of $G$} is a function $\kappa: V \to \bbZ_{> 0}$ such that $\kappa(v) \neq \kappa(w)$ if $\{v, w\} \in E$.
In 1995, Stanley \cite{StanleyChromatic} defined the following family of symmetric functions.
\begin{defin}
\cite{StanleyChromatic}
Let $G = (V, E)$ be a graph. The \emph{chromatic symmetric function} $X_G(\bx)$ of $G$ is defined by
\begin{equation}
X_G(\bx) = \sum_{\kappa} \prod_{v \in V} x_{\kappa(v)},
\end{equation}
where the sum is over all proper colorings $\kappa: V \to \mathbb{Z}_{> 0}$ of $G$.
\end{defin}

Chromatic symmetric functions are particularly interesting in the case when $G$ is the incomparability graph $\inc(P)$ of a \threeone -free poset $P$. For a poset $P$, the \emph{incomparability graph} $\inc(P)$ is the graph with vertex set $P$ and edges between each pair of incomparable elements of $P$. A poset $P$ is \threeone -free if it contains no induced subposet which is the disjoint union of a 3-chain and a 1-chain. In this case, Gasharov \cite{Gasharov} defined $P$-tableaux and showed that they give a combinatorial interpretation of the Schur-coefficients of $X_{\inc(P)}(\bx)$.
\begin{defin}
\cite{Gasharov}
Let $P$ be a \threeone -free poset, and let $\lambda$ be a partition. A filling $T$ of $\lambda$ with entries in $P$ is a \emph{$P$-tableau} if $T_{i,j} <_P T_{i+1,j}$ and $T_{i,j} \not>_P T_{i, j+1}$ for all such adjacent pairs of cells in $\lambda$. Equivalently, the entries of $T$ are increasing in $P$ down columns and non-descending in $P$ across rows. Write $\T_P(\lambda)$ for the set of $P$-tableaux of shape $\lambda$, and write $\ST_P(\lambda)$ for the set of $P$-tableaux in which each element of $P$ appears exactly once. Note that this definition is the transpose of Gasharov's $P$-tableaux. We take this convention so that $\T_P(\lambda) = \SSYT(\lambda, n)$ when $P$ is the total order on $[n]$.
\end{defin}
\begin{theorem}\cite{Gasharov}
\label{thm Gasharov}
Let $P$ be a \threeone -free poset. Then $X_{\inc(P)}(\bx) = \sum_{\lambda \vdash |P|} |\ST_P(\lambda)| s_{\lambda'}(\bx)$.
\end{theorem}

In 1993, Stanley--Stembridge conjectured that the chromatic symmetric functions of incomparability graphs of \threeone -free posets are $e$-positive \cite{StanleyStembridge}. This conjecture was resolved more than 30 years later by Hikita.

\begin{theorem}\cite{HikitaChromatic}
Let $P$ be a \threeone -free poset. Then $X_{\inc(P)}(\bx)$ is $e$-positive. 
\end{theorem}

An important case of \threeone -free posets are those coming from Hessenberg functions. A sequence $\bh = (h_1, h_2,\dots, h_n) \in [n]^n$ is a \emph{Hessenberg function} if $i \leq h_i \leq h_{i+1}$ for all $i$. Write $\mathbb{H}_n$ for the set of Hessenberg functions of length $n$. The natural unit interval order $P_{\bh}$ associated to $\bh \in \mathbb{H}_n$ is the poset on $[n]$ such that $i <_{P_\bh} j$ if $h_i < j.$ Guay-Paquet showed that if $P$ is a \threeone -free poset, then $X_{\inc(P)}(\bx)$ is a convex combination of chromatic symmetric functions of incomparability graphs of natural unit interval orders \cite{GPchromatic}. Hence, positivity problems for \threeone -free posets reduce to positivity of natural unit interval orders.
Shareshian--Wachs \cite{SWchromatic} introduced the \emph{chromatic quasisymmetric function} $X_G(\bx, q)$ of a graph with vertex set $[n]$, and showed that $X_G(\bx, q)$ is in fact a symmetric function when $G = \inc(P_\bh)$. Shareshian--Wachs \cite{SWchromatic} conjectured, and Brosnan--Chow \cite{BrosnanChow} and Guay-Paquet \cite{guaypaquetCohom} proved that $\omega X_{\inc(P_\bh)}(\bx, q)$ is the graded Frobenius characteristic of the $S_n$ action on the equivariant cohomology of the regular semisimple Hessenberg variety associated to $\bh$, where $\omega$ is the involution on symmetric functions sending $s_\lambda(\bx)$ to $s_{\lambda'}(\bx)$. The Shareshian--Wachs conjecture on the $e$-positivity of $X_{\inc(P_\bh)}(\bx, q)$ was recently proven by OpenAI \cite{OAIePos}.

Another interesting class of graphs are the \emph{claw-free} graphs. A graph $G$ is \emph{claw-free} if it has no induced $K_{1,3}$. Gasharov \cite{GasharovClaw, StanleyChromatic1998} conjectured that if $G$ is claw-free, then $X_G(\bx)$ is Schur-positive. 
Minimal counterexamples to this conjecture were recently discovered independently by Matherne--Morales \cite{MMclaw} and Prajapati \cite{PrajapatiClaw}, and infinite families of counterexamples were constructed by Wang--Zhang--Zhao \cite{WZZclaw}. 
The minimal counterexamples to Gasharov's conjecture are shown in Figure~\ref{fig non s pos claw free graphs}.

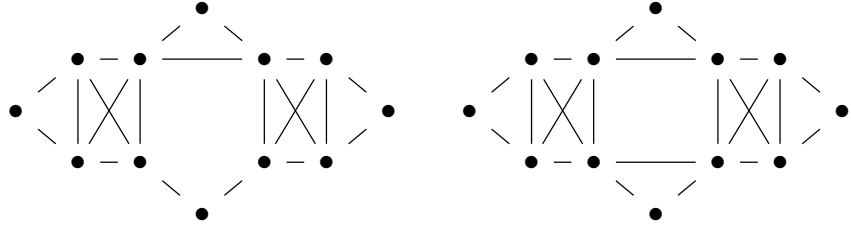
\begin{figure}[H]
\caption{Non-Schur-Positive Claw-Free Graphs}
\label{fig non s pos claw free graphs}
\[\begin{tikzcd}[ampersand replacement=\&,sep=tiny]
	\&\&\& \bullet \&\&\&\&\&\&\&\& \bullet \&\&\& \\
	\& \bullet \& \bullet \&\& \bullet \& \bullet \&\&\&\& \bullet \& \bullet \&\& \bullet \& \bullet \\
	\bullet \&\&\&\&\&\& \bullet \&\& \bullet \&\&\&\&\&\& \bullet \\
	\& \bullet \& \bullet \&\& \bullet \& \bullet \&\&\&\& \bullet \& \bullet \&\& \bullet \& \bullet \\
	\&\&\& \bullet \&\&\&\&\&\&\&\& \bullet
	\arrow[no head, from=1-4, to=2-5]
	\arrow[no head, from=1-12, to=2-13]
	\arrow[no head, from=2-2, to=2-3]
	\arrow[no head, from=2-2, to=4-3]
	\arrow[no head, from=2-3, to=1-4]
	\arrow[no head, from=2-3, to=4-3]
	\arrow[no head, from=2-5, to=2-3]
	\arrow[no head, from=2-5, to=4-5]
	\arrow[no head, from=2-5, to=4-6]
	\arrow[no head, from=2-6, to=2-5]
	\arrow[no head, from=2-6, to=3-7]
	\arrow[no head, from=2-10, to=2-11]
	\arrow[no head, from=2-10, to=4-11]
	\arrow[no head, from=2-11, to=1-12]
	\arrow[no head, from=2-11, to=4-11]
	\arrow[no head, from=2-13, to=2-11]
	\arrow[no head, from=2-13, to=4-13]
	\arrow[no head, from=2-13, to=4-14]
	\arrow[no head, from=2-14, to=2-13]
	\arrow[no head, from=2-14, to=3-15]
	\arrow[no head, from=3-1, to=2-2]
	\arrow[no head, from=3-1, to=4-2]
	\arrow[no head, from=3-7, to=4-6]
	\arrow[no head, from=3-9, to=2-10]
	\arrow[no head, from=3-9, to=4-10]
	\arrow[no head, from=3-15, to=4-14]
	\arrow[no head, from=4-2, to=2-2]
	\arrow[no head, from=4-2, to=2-3]
	\arrow[no head, from=4-3, to=4-2]
	\arrow[no head, from=4-3, to=5-4]
	\arrow[no head, from=4-5, to=2-6]
	\arrow[no head, from=4-5, to=4-6]
	\arrow[no head, from=4-6, to=2-6]
	\arrow[no head, from=4-10, to=2-10]
	\arrow[no head, from=4-10, to=2-11]
	\arrow[no head, from=4-11, to=4-10]
	\arrow[no head, from=4-11, to=4-13]
	\arrow[no head, from=4-11, to=5-12]
	\arrow[no head, from=4-13, to=2-14]
	\arrow[no head, from=4-13, to=4-14]
	\arrow[no head, from=4-14, to=2-14]
	\arrow[no head, from=5-4, to=4-5]
	\arrow[no head, from=5-12, to=4-13]
\end{tikzcd}\]
\end{figure}

In 1998, Stanley introduced \emph{$G$-analogues} of symmetric functions as a tool to study chromatic symmetric functions. To avoid confusion with the $F$-analogues defined in the introduction, we will always use $F$ to denote a polynomial and $G$ to denote a graph. 
\begin{defin} \cite{StanleyChromatic1998}
Let $G$ be a graph with vertex set $[n]$, and let $z_1,z_2,\dots,z_n$ be commuting indeterminates. For a partition $\lambda$, define the \emph{elementary $G$-analogue} $e_{\lambda}^G(\bu)$ by
\begin{equation}
e_k^G(\bz) = \sum_{\{i_1,i_2,\dots,i_k\} \text{ indep. in } G} z_{i_1}z_{i_2}\cdots z_{i_k} \text{ \ \ and \ \ } e_{\lambda}^G(\bz) = e_{\lambda_1}^G(\bz) e_{\lambda_2}^G(\bz) \cdots e_{\lambda_\ell}^G(\bz).
\end{equation}
For a symmetric function $f(\bx) = \sum_{\lambda} a_\lambda e_{\lambda}(\bx)$, define its $G$-analogue $f^G(\bz)$ by $f^G(\bz) = \sum_{\lambda} a_{\lambda} e_\lambda^G(\bz)$.
As with $F$-analogues, we obtain the $G$-analogue of the dual Cauchy product identity, which we denote $\Omega^G(\bz, \bx)$ and define by
\begin{align}
\Omega^G(\bz, \bx) &= \prod_{i \geq 1} \sum_{k \geq 0} e_k^G(\bz) x_i^k = \sum_{\lambda} e_\lambda^G(\bz) m_\lambda(\bx) \\
&= \sum_\lambda (-1)^{|\lambda| - \ell(\lambda)} z_\lambda^{-1} p_\lambda^G(\bz) p_\lambda(\bx) = \sum_{\lambda} s_\lambda^G(\bz) s_{\lambda'}(\bx) = \sum_{\lambda} m_\lambda^G(\bz) e_\lambda(\bx).
\end{align}
\end{defin}
For a graph $G$ with vertex set $[n]$ and a composition $\alpha \in \mathbb{Z}_{\geq 0}^n$, we write $G[\alpha]$ for the graph obtained by replacing vertex $i$ of $G$ with a clique of size $\alpha_i$. See Figure~\ref{fig blowup}. It is straightforward to verify that if $G$ is the incomparability graph of a \threeone -free poset, then $G[\alpha]$ is also the incomparability graph of a \threeone -free poset. Likewise, if $G$ is claw-free, then $G[\alpha]$ is also claw-free.

\begin{figure}[H]
\caption{}
\label{fig blowup}
\[\begin{tikzcd}[ampersand replacement=\&]
	\&\&\&\&\&\&\& \bullet \\
	\&\&\&\&\& \bullet \&\& \bullet \& \bullet \\
	1 \& 2 \& 3 \& 4 \&\& \bullet \& \bullet \& \bullet \& \bullet \\
	\&\& G \&\&\&\& {G[2,1,3,2]}
	\arrow[no head, from=1-8, to=2-8]
	\arrow[no head, from=1-8, to=2-9]
	\arrow[curve={height=6pt}, no head, from=1-8, to=3-8]
	\arrow[no head, from=1-8, to=3-9]
	\arrow[no head, from=2-6, to=3-6]
	\arrow[no head, from=2-6, to=3-7]
	\arrow[no head, from=2-8, to=2-9]
	\arrow[no head, from=2-8, to=3-8]
	\arrow[no head, from=2-8, to=3-9]
	\arrow[no head, from=3-1, to=3-2]
	\arrow[no head, from=3-2, to=3-3]
	\arrow[no head, from=3-3, to=3-4]
	\arrow[no head, from=3-6, to=3-7]
	\arrow[no head, from=3-7, to=1-8]
	\arrow[no head, from=3-7, to=2-8]
	\arrow[no head, from=3-7, to=3-8]
	\arrow[no head, from=3-8, to=2-9]
	\arrow[no head, from=3-8, to=3-9]
	\arrow[no head, from=3-9, to=2-9]
\end{tikzcd}\]
\end{figure}

Stanley \cite{StanleyChromatic1998} expressed $\Omega^G(\bz, \bx)$ as a weighted sum of chromatic symmetric funtions. As a consequence, positivity of $\Omega^G(\bz, \bx)$ in a given basis is equivalent to positivity of $X_{G[\alpha]}(\bx)$ for all $\alpha \in \bbZ_{\geq 0}^n$.
Using results of Guay-Paquet \cite{GPchromatic} and Hikita's proof \cite{HikitaChromatic} of the Stanley--Stembridge conjecture, we obtain the following corollary.
\begin{theorem}\cite{StanleyChromatic}
\label{thm OmegaG sum}
Let $G$ be a graph on $[n]$. Then
\begin{equation}
\Omega^G(\bz, \bx) = \sum_{\alpha \in \mathbb{Z}_{\geq 0}^n} \frac{\bz^{\alpha}}{\alpha!} X_{G[\alpha]}(\bx),
\end{equation}
where $\alpha! = \alpha_1!\alpha_2!\cdots \alpha_n!$.
Hence, $X_G(\bx)$ is the coefficient of $z_1z_2\cdots z_n$ in $\Omega^G(\bz, \bx)$.
\end{theorem}
\begin{corollary}
If $G$ is the incomparability graph of a \threeone -free poset, then $\Omega^G(\bz, \bx)$ is $e$-positive.
\end{corollary}

For a graph $G$, define the \emph{independence polynomial} $I_G(t)$ by
\begin{equation}
I_G(t) = \sum_{S \text{ indep. in } G} t^{|S|},
\end{equation}
and define the \emph{multivariate independence polynomial} $I_G(\bz, t)$ by
\begin{equation}
I_G(\bz, t) = \sum_{S \text{ indep. in } G} t^{|S|} \prod_{i \in S} z_i = \sum_{k \geq 0} e_k^G(\bz) t^k.
\end{equation}
Note that $\Omega^G(\bz, \bx) = \Omega_{I_G}(\bz, \bx)$, using notation from \eqref{eq Omega F def}. 
As a corollary of Theorem~\ref{thm stan poly s pos}, Theorem~\ref{thm Gasharov}, and Theorem~\ref{thm OmegaG sum}, Stanley observed that if $G$ is the incomparability graph of a \threeone -free poset, then all roots of $I_G(\br, t)$ are negative real numbers for all $\br \in \mathbb{R}_{\geq 0}^n$ \cite{StanleyChromatic1998}. Chudnovsky--Seymour \cite{ChudnovSeym} then showed that $I_G(t)$ is real-rooted for all claw-free graphs $G$. Engstr\"om \cite{Engstrom} and Leake--Ryder \cite{LeakeRyder} showed that $I_G(\br, t)$ is real-rooted for all claw-free graphs $G$ and $\br \in \mathbb{R}_{\geq 0}^n$.


Posets coming from Hessenberg functions are a special case of \emph{interval orders}. For a set of intervals $J = \{(a_i < b_i) \mid i \in [n]\}$ define the associated interval order $P_J$ to be the  
poset such that $i <_P j$ if $b_i < a_j$. An interval order $P$ is \emph{naturally labelled} if $i <_P j$ implies $i < j$ as integers. An interval order $P$ is a \emph{natural unit interval order} if $P$ is naturally labeled and $b_i = a_i + 1$ for all $i \in [n]$. 
Posets corresponding to Hessenberg functions are exactly the natural unit interval orders. 

A \emph{matching} $M$ in a graph $G$ is subset of edges of $G$ such that no two edges in $M$ share a vertex.
A matching $M$ in a graph $G$ is a \emph{perfect matching} if each vertex is contained in an edge of $M$. Write $\mathcal{PM}(G)$ for the set of perfect matchings of $G$. 
For a graph $G$ with vertex set $V(G) \subset \bbR$, treat each edge as an interval in $\bbR$, and define the symmetric function $PM_G(\bx)$ by
\begin{equation}
PM_G(\bx) = \sum_{M \in \mathcal{PM}(G)} X_{\inc(P_M)}(\bx).
\end{equation} 
Amdeberhan--Shareshian--Stanley \cite{AmdShaSta} proved the following positivity result for $PM_{K_{2n}}(\bx)$.
\begin{theorem}\cite{AmdShaSta}
\label{thm AmdShaSta match}
Let $F(t) = \sec(\sqrt{t})$. Then
\begin{equation}
\omega \Omega_F(\bx) = \sum_{n \geq 0} \frac{1}{(2n)!}PM_{K_{2n}}(\bx),
\end{equation}
where $K_{2n}$ is the complete graph on $[2n]$. Furthermore, $PM_{K_{2n}}$ is $e$-positive for all $n$.
\end{theorem}
Alexandersson--Dai \cite{AlexanderssonDai} considered the homomorphisms $T_k: \Lambda \to \Lambda$ on symmetric functions defined by $T_k: e_i(\bx) \mapsto e_{ki}(\bx)$, and showed that the coefficient of $e_\lambda(\bx)$ in $PM_{K_{2n}}(\bx)$ is the coefficient of $m_{2n}(\bx)$ in $T_2(m_\lambda(\bx))$. This inspired the following conjecture.
\begin{conjecture}\cite{AlexanderssonDai}
\label{conj AlexanderssonDai}
For all partitions $\lambda$, $T_k(m_\lambda(\bx))$ is monomial positive.
\end{conjecture}

\subsection{Majorization}

In this subsection, we set notation for several relations on real vectors and recall some useful facts. In particular, we recall \emph{majorization} and \emph{log-majorization}. An extensive treatment of majorization can be found in the book by Marshall, Olkin, and Arnold \cite{MOAmaj}.
The results of this subsection will be used in Section~\ref{sec sym evals}.
  
\begin{defin}
Let $\bu, \bv \in \mathbb{R}^n$. Write $u_{[i]}$ for the $i^{th}$ largest entry of $\bu$. We say $\bu$ \emph{weakly majorizes} $\bv$ and write $\bu \succeq_w \bv$ if, for each $1 \leq k \leq n$,
\begin{equation}
\sum_{i = 1}^k u_{[i]} \geq \sum_{i = 1}^k v_{[i]}.
\end{equation}
If $\bu \succeq_w \bv$ and $\sum_{i = 1}^n u_{[i]} = \sum_{i=1}^n v_{[i]}$, we say $\bu$ \emph{majorizes} $\bv$ and write $\bu \succeq \bv$.
For $\bu, \bv \in \mathbb{R}_{\geq 0}^n$, we say $\bu$ \emph{log-majorizes} $\bv$ and write $\bu \succeq_{\log} \bv$ if, for each $1 \leq k \leq n$,
\begin{equation}
\prod_{i=1}^k u_{[i]} \geq \prod_{i = 1}^k v_{[i]}.
\end{equation}
Equivalently, $\bu \succeq_{\log} \bv$ if $(\log u_1, \log u_2,\dots, \log u_n) \succeq_w (\log v_1, \log v_2, \dots, \log v_n)$.
\end{defin}

For any non-decreasing convex function $f: \mathbb{R} \to \mathbb{R}$, applying $f$ entrywise preserves weak majorization. Hence, if $\bu \succeq_{\log} \bv$, then $\bu \succeq_w \bv$.
The following lemmas of Ando \cite{Ando1989} and Dean--Verducci \cite{DeanVerducci1990} are important tools for our proof of Theorem~\ref{thm sym eval ineq}.
For $\bu, \bv \in \mathbb{R}^n$, write $\bu \leq_{\mathrm{entry}} \bv$ if $u_i \leq v_i$ for each $1 \leq i \leq n$.
\begin{lemma}\cite[Corollary 1.8]{Ando1989}
\label{lem Ando weak maj}
Let $\bu, \bv \in \mathbb{R}^n$. Then $\bu \preceq_w \bv$ if and only if there exists $\bw \in \mathbb{R}^n$ such that $\bu \leq_{\mathrm{entry}} \bw$ and $\bw \preceq \bv$.
\end{lemma}

\begin{lemma}\cite{DeanVerducci1990}
\label{lem DV linear pres}
Let $T: \mathbb{R}^n \to \mathbb{R}^m$ be a linear map. Then $T\bu \preceq T\bv$ for all $\bu, \bv \in \mathbb{R}^n$ such that $\bu \preceq \bv$ if and only if for any $n \times n$ permutation matrix $P$, there exists an $m \times m$ permutation matrix $Q$ such that $TP = QT$.
\end{lemma}

Majorization is closely related to symmetric polynomials via \emph{symmetric means}. For a symmetric polynomial $g(\bx)$ in $n$ variables, define the associated symmetric mean $M[g](\bx) = \left(\frac{g(\bx)}{g(\mathbf{1}_n)}\right)^{1/n}$. 
A classical result of Muirhead \cite{Muirhead} states that for partitions $\lambda, \mu \vdash k$, we have $M[m_\lambda](\bx) \geq M[m_\mu](\bx)$ for all $\bx \in \mathbb{R}_{\geq 0}^n$ if and only if $\lambda \succeq \mu$. 
Cutler--Greene--Skandera \cite{CGSsymMeans} conjectured and Sra \cite{Sra} proved that for $\lambda, \mu \vdash k$, $M[s_\lambda](\bx) \geq M[s_\mu](\bx)$ for all $\bx \in \mathbb{R}_{\geq 0}^n$ if and only if $\lambda \succeq \mu$.
Khare--Tao \cite{KhareTao} generalized Sra's theorem by defining generalized Schur symmetric means indexed by positive real vectors.

\subsection{Total Positivity} 
\label{sec LGV background}
Let $A$ be a square matrix with real number entries. We say $A$ is \emph{totally positive} (totally nonnegative) if every minor is positive (nonnegative). For a matrix $A$ with entries in a $\mathbb{Q}$-algebra $\mathcal{R}$ such that entries of $A$ commute, we say $A$ is totally positive with respect to a positive cone $\mathcal{R}_+$ if all minors of $A$ are in $\mathcal{R}_+$. Totally positive matrices are closely related to networks via Lindstr\"om's lemma. 
In this subsection, we recall Lindstr\"om's lemma and related results on totally nonnegative matrices. Additionally, we recall the \emph{cylindric partitions} of Gessel--Krattenthaler \cite{GKcyl} and extend their results to general \emph{shift-symmetric} networks. The results of this subsection will be used in Section~\ref{sec s pos from networks}. 
\begin{defin}
Let $\Gamma$ be a directed acyclic graph. Let $\bs = \bs^\Gamma = (s_i)$ and $\bt = \bt^\Gamma = (t_i)$ be totally ordered collections of source and sink vertices in $\Gamma$. Write $P(s_i, t_j)$ for the set of paths from $s_i$ to $t_j$. We say $\Gamma$ is \emph{compatible} with $\bs$ and $\bt$ if $\gamma \in P(s_i, t_j)$ and $\eta \in P(s_k, t_l)$ share a vertex whenever $i < k$ and $j > l$. Throughout the paper, we assume that each directed acyclic graph has compatible source and sink labels.

Let $\wt: E(\Gamma) \to \mathcal{R}$ be the \emph{weight function of $\Gamma$}. We assume the edge weights $\{\wt(e) \mid e \in E(\Gamma)\}$ commute pairwise. For a path $\gamma \in P(s_i, t_j)$, let $\wt(\gamma) = \prod_{e \in \gamma} \wt(e)$, and for a family of paths $A = (\gamma_1, \gamma_2,\dots, \gamma_k)$, let $\wt(A) = \prod_{i = 1}^k \wt(\gamma_i)$. For a pair $(s_i, t_j)$, let $\wt(s_i, t_j) = \sum_{\gamma \in P(s_i, t_j)} \wt(\gamma)$.
We say a weighted, directed, acyclic graph $\Gamma$ with compatible collections of source and sink vertices is a \emph{network}.

Let $W_\Gamma = \left[\wt(s_i, t_j)\right]$ be the \emph{weight matrix of $\Gamma$}.
For $\ba = (a_1,a_2,\dots, a_k)$ and $\bb = (b_1,b_2,\dots, b_k)$, let $\bP(\ba, \bb)$ be the collection of all families of paths $(\gamma_1, \gamma_2,\dots, \gamma_k)$ such that $\gamma_r$ is a path from $s_{a_r}$ to $t_{b_r}$.
Without loss of generality, we assume that the source vertex indices $\ba = (a_1, a_2,\dots,a_k)$ are always in increasing order.
Let $\bN(\ba, \bb) \subset \bP(\ba,\bb)$ be the set of path families such that no two paths intersect.
\end{defin}

The following result, known as Lindstr\"om's Lemma, gives a correspondence between minors of the weight matrix of $\Gamma$ and noncrossing path families in $\Gamma$. As a corollary, we have that if the weights of $\Gamma$ are positive, then the weight matrix of $\Gamma$ is totally positive. Brenti \cite{BrentiPF} proved the converse of this result for totally nonnegative matrices with real entries, showing that for any such matrix $A$, there exists a graph $\Gamma$ with nonnegative weights such that $A = W_\Gamma$. 

\begin{theorem}[Lindstr\"om's Lemma \cite{GesselViennot, Lindstrom}] 
\label{thm Lindstrom}
Let $\Gamma$ be a network, let $I = \{i_1 <  i_2 < \dots < i_k\}$, and let $J = \{j_1 < j_2 < \dots < j_k\}$. Then the $I, J$-minor of the weight matrix satisfies
\begin{equation}
\Delta_{I, J} W_\Gamma = \det(\wt(s_i, t_j))_{i \in I, j \in J} = \sum_{A \in \bN(I, J)} \wt(A).
\end{equation}
\end{theorem}
\begin{corollary}
Let $\Gamma$ be a directed acyclic graph, and let $\bs$, $\bt$ be a pair of source and sink labels compatible with $\Gamma$. If $\wt(e) \in \mathcal{R}_+$ for all $e \in E(\Gamma)$, then the weight matrix $W_\Gamma$ is totally nonnegative with respect to $\mathcal{R}_+$.
\end{corollary}

\begin{theorem}\cite{BrentiPF}
Let $A$ be a totally nonnegative matrix with real entries. Then there exists a network $\Gamma$ with nonnegative edge weights such that $A = W_\Gamma$.
\end{theorem}

The following lemma shows that $(\gamma_1, \gamma_2, \dots, \gamma_l) \in \bP(\ba, \bb)$ with $a_1 < a_2 < \dots < a_l$ is intersecting if and only if there is an adjacent pair of paths $\gamma_i$ and $\gamma_{i+1}$ which intersect. This fact allows us to more easily characterize non-intersecting path families. The proof is straightforward.

\begin{lemma}
\label{lem adj path intersect}
Let $\ba = (a_1 < a_2 < \dots < a_l)$ be a sequence of source vertex indices, and let $\bb$ be a sequence of sink vertex indices. Let $(\gamma_1, \gamma_2, \dots, \gamma_l) \in \bP(\ba, \bb)$. If $\gamma_k$ intersects $\gamma_j$ for some $k, j$, then there is an index $i$ such that $\gamma_i$ intersects $\gamma_{i+1}$.  
\end{lemma}

In this paper, we will be primarily interested in totally nonnegative Toeplitz matrices. A network $\Gamma$ is \emph{shift-symmetric} with respect to source and sink labels $\bs = (s_i)_{i \in \mathbb{Z}}$ and $\bt = (t_j)_{j \in \mathbb{Z}}$ if $\wt(s_i, t_j) = \wt(s_{i+k}, t_{j+k})$ for all $i,j, k \in \mathbb{Z}$. Equivalently, if $\Gamma$ is shift-symmetric, then $W_\Gamma$ is a Toeplitz matrix. The motivating example of a shift-symmetric network can be found in the Gessel--Viennot \cite{GesselViennot} proof of the Jacobi--Trudi identity, giving a network with weight matrix $W = [e_{j-i}(\bx)]_{i,j = 1}^\infty$. More examples of totally nonnegative Toeplitz matrices can be found in the paper \cite{BrentiPF} by Brenti.

In many cases, the shift-symmetry of $\Gamma$ will be induced by a \emph{shift automorphism} $\mathsf{S}: \Gamma \to \Gamma$ which sends each source $s_i$ to $s_{i+1}$, sends each sink $t_j$ to $t_{j+1}$, and preserves the weight of each edge. For such cases, we are interested in cylindrically nonintersecting path families.
For integer sequences $\ba = (a_1 < a_2 < \dots < a_k)$ and $\bb = (b_1 < b_2 < \dots < b_k),$ and a \emph{period} $d > \max\{a_k - a_1, b_k - b_1\}$, a path family $(\gamma_1, \gamma_2, \dots, \gamma_k) \in \bP(\ba,\bb)$ is \emph{cylindrically nonintersecting} with period $d$ if the path family $(\mathsf{S}^{md} \gamma_i \mid i,m \in \mathbb{Z})$ is nonintersecting. Write $\bCN(\ba, \bb, d)$ for the set of cylindrically nonintersecting path families from $(s_i)_{i \in \ba}$ to $(t_j)_{j \in \bb}$ with period $d$.
Write $\bbZ^k_0$ for the set of $\br \in \bbZ^k$ such that $\sum_{i = 1}^k r_i = 0$.
For $\sigma \in S_k$ and $\br \in \bbZ^k_0$, let $\sigma(\bb) + \br d = (b_{\sigma(1)} + r_1 d, b_{\sigma(2)} + r_2 d, \dots, b_{\sigma(k)} + r_k d)$.

\begin{lemma}
\label{lem cyl nonint paths}
Let $\ba = (a_1 < a_2 < \dots < a_k)$ and $\bb = (b_1 < b_2 < \dots < b_k)$ be integer sequences, and let $d > \max\{a_k - a_1, b_k - b_1\}$. Let $\sigma \in S_k$ and $\br \in \bbZ^k_0$. 
If $\sigma$ is not the identity permutation or if $\br \neq \mathbf{0},$ then $\bN(\ba, \sigma(\bb) + \br d) = \emptyset.$ 
\end{lemma}
\begin{proof}
Let $(\gamma_1, \gamma_2,\dots, \gamma_k)$ be a path family in $\bP(\ba, \sigma(\bb) + \br d)$.
Suppose $\br \neq \mathbf{0}$. 
Then there are indices $i, j$ such that $r_i > 0$ and $r_j < 0$.
If $i < j$, then, as $d > b_k - b_1$, we have $b_{\sigma(i)} + r_i d > b_{\sigma(j)} + r_j d$.
As $\gamma_i$ is a path from $s_{a_i}$ to $t_{b_{\sigma(i)} + r_i d}$ and $\gamma_j$ is a path from $s_{a_j}$ to $t_{b_{\sigma(j)} + r_j d}$ with $a_i< b_j$, this means that $\gamma_i$ and $\gamma_j$ intersect. 
If $i > j$,  then as $d > b_k - b_1$, we have $b_{\sigma(i)} + r_i d > b_{\sigma(j)} + (r_j + 1) d$.
As $d > a_k - a_1$, we have $a_i < a_j + d$. Hence, as $\gamma_i$ is a path from $s_{a_i}$ to $t_{b_{\sigma(i)} + r_i d}$ and $\mathsf{S}^d \gamma_j$ is a path from $s_{a_j + d}$ to $t_{b_{\sigma(j)} + (r_j + 1) d}$, this means that $\gamma_i$ and $\mathsf{S}^d \gamma_j$ intersect.

Suppose $\br = \mathbf{0}$, and $\sigma$ is not the identity permutation. Then there is some $i < j$ such that $\sigma(i) > \sigma(j)$, so $\gamma_i$ intersects $\gamma_j$.
\end{proof}
The following theorem is a slight generalization of a result of Gessel--Krattenthaler \cite{GKcyl}. For a path family $A$, define the \emph{skeleton} $\skel(A)$ of $A$ to be the multiset of edges in $A$. 
\begin{theorem}\cite{GKcyl}
\label{thm GK det sum}
Let $\Gamma$ be a network with shift automorphism $\mathsf{S}$, let $\ba = (a_1 < a_2 < \dots < a_k)$ and $\bb = (b_1 < b_2 < \dots < b_k)$ be integer sequences, and let $d > \max\{a_k - a_1, b_k - b_1\}$. Suppose that for all vertices $v$ of $\Gamma$, there is no path from $v$ to $\mathsf{S}^{md}v$ for all $m \in \mathbb{Z}$. Then 
\begin{equation}
\sum_{A \in \bCN(\ba,\bb,d)} \wt(A) = \sum_{\br \in \bbZ^k_0} \det\left[ \wt(s_{a_\alpha}, t_{b_{\beta} + r_\alpha d})\right]_{\alpha, \beta = 1}^k.
\end{equation}
\end{theorem} 
\begin{proof}
Given a path family $A \in \bP(\ba, \sigma(\bb) + \br d)$, the permutation $\sigma$ can be recovered as $d > b_k - b_1$. Thus, we may define $\sgn(A) = \sgn(\sigma)$ for $A \in \bP(\ba, \sigma(\bb) + \br d)$.
Consider the set of path families
\begin{equation}
\mathcal{P} = \bigsqcup_{\br \in \mathbb{Z}^k_0} \bigsqcup_{\sigma \in S_k} \bP(\ba, \sigma(\bb) + \br d).
\end{equation}
Observe that 
\begin{align}
\sum_{\br \in \bbZ^k_0} \det\left[ \wt(s_{a_\alpha}, t_{b_{\beta} + r_\alpha d})\right]_{\alpha, \beta = 1}^k &= \sum_{\br \in \bbZ^k_0}  \sum_{\sigma \in S_k} \sum_{A \in \bP(\ba, \sigma(\bb) + \br d)} \sgn(\sigma) \wt(A) \\
&= \sum_{A \in \mathcal{P}} \sgn(A) \wt(A).
\end{align} 
 
Let $\mathcal{P}^-$ denote the set of path families in $\mathcal{P}$ which are cylindrically intersecting with period $d$.
By Lemma~\ref{lem cyl nonint paths}, $\bCN(\ba, \bb, d) = \mathcal{P} \setminus \mathcal{P}^-$. 
Note that $A \in \mathcal{P}$ is cylindricly intersecting with period $d$ if and only if there is a vertex $v$ of $\Gamma$ such that $\skel(\mathsf{S}^{\mathbb{Z} d} A)$ contains two edges directed into $v$. 
Let $\skel(\mathcal{P}) = \{ \skel(\mathsf{S}^{\mathbb{Z} d} A) \mid A \in \mathcal{P}\}$ denote the set of skeletons of periodic path families, and let $\skel(\mathcal{P}^{-}) \subset \skel(\mathcal{P})$ denote the set of intersecting cylindric skeletons. 
For each $\mathcal{K} \in \skel(\mathcal{P}^-)$, fix a pair of edges $X_{\mathcal{K}}$ which are directed into the same vertex.

For $A = (\gamma_1, \gamma_2, \dots, \gamma_k) \in \mathcal{P}^-$, let $\mathcal{K} = \skel(\mathsf{S}^{\mathbb{Z}d} A) \in \skel(\mathcal{P}^-)$. 
As $A \in \mathcal{P}^-$, there are paths $\mathsf{S}^{m_1 d}\gamma_i$ and $\mathsf{S}^{m_2 d} \gamma_j$ which intersect at $X_{\mathcal{K}}$. Write $\gamma_i = (\gamma_{i,1}, \gamma_{i,2}, \dots, \gamma_{i, l_i})$ and $\gamma_j = (\gamma_{j,1}, \gamma_{j,2}, \dots, \gamma_{j, l_j})$ for the sequence of edges in $\gamma_i$ and $\gamma_j$, and let $\gamma_{i, v}$ and $\gamma_{j, w}$ be the edges such that $\{\mathsf{S}^{m_1d} \gamma_{i,v}, \mathsf{S}^{m_2 d} \gamma_{j, w}\} = X_{\mathcal{K}}$. 
Define $\varphi(A) = (\gamma_1', \gamma_2', \dots, \gamma_k')$ to be the path family such that 
\begin{equation}
\gamma_\alpha' = \begin{cases}
\gamma_\alpha & \text{if } \alpha \not\in \{i, j\}, \\
(\gamma_{i,1}, \gamma_{i,2}, \dots, \gamma_{i, v}, \mathsf{S}^{(m_2 - m_1)d} \gamma_{j, w+1},\mathsf{S}^{(m_2 - m_1)d} \gamma_{j, w+2}, \dots, \mathsf{S}^{(m_2 - m_1)d} \gamma_{j, l_j}) & \text{if } \alpha = i, \\
(\gamma_{j,1}, \gamma_{j,2}, \dots, \gamma_{j, w}, \mathsf{S}^{(m_1 - m_2)d} \gamma_{i, v+1},\mathsf{S}^{(m_1 - m_2)d} \gamma_{i, v+2}, \dots, \mathsf{S}^{(m_1 - m_2)d} \gamma_{i, l_i}) & \text{if } \alpha = j.
\end{cases} 
\end{equation}
Observe that $\varphi(A)$ is a path family from $\ba$ to $\pi(\bb) + \br'$, where $\pi = \tau_{i,j} \circ \sigma$ and $\br' = \tau_{i,j} \br + (m_1 - m_2)(\mathbf{1}_j - \mathbf{1}_i)$, where $\tau_{i,j}$ is the transposition swapping $i$ and $j$, and $\mathbf{1}_i$ is the $i^{th}$ standard basis vector. 
As there is no path from $v$ to $\mathsf{S}^{md}v$ for any $m \in \mathbb{Z}$, we have that $i \neq j$ and $\sgn(\pi) = -\sgn(\sigma)$.  
Observe that $\varphi$ is an involution and that $\wt(A) = \wt(\varphi(A))$.
By applying $\varphi$ to each $A \in \mathcal{P}^-$, we have
\begin{equation}
\sum_{A \in \mathcal{P}^-} \sgn(A) \wt(A) = - \sum_{A \in \mathcal{P}^-} \sgn(A) \wt(A) = 0.
\end{equation}
By Lemma~\ref{lem cyl nonint paths}, we have that $\sgn(A) = 1$ for all $A \in \bCN(\ba, \bb, d)$, so
\begin{equation}
\sum_{A \in \mathcal{P}} \sgn(A) \wt(A) = \sum_{A \in \bCN(\ba, \bb, d)} \wt(A).
\end{equation}
\end{proof}

Gessel--Krattenthaler used Theorem~\ref{thm GK det sum} to count \emph{cylindric tableaux}. Let $\lambda$ be a partition with $c$ columns, let $\mu$ be partition such that $\mu \subset \lambda$, and let $d > \max\{\lambda_1' - \lambda_c', \mu_1' - \mu_c'\}$ be an integer. A semistandard Young tableau $T \in \SSYT(\lambda/\mu)$ is a \emph{cylindric tableau} of \emph{cylindric shape} $\lambda/\mu/d$ if $T_{r, c} \leq T_{r+d, 1}$ for all $\mu_c' < r \leq \lambda_c'$. Write $\CSSYT(\lambda/\mu/d)$ for the set of all cylindric tableaux of cylindric shape $\lambda/\mu/d$.
Gessel--Krattenthaler \cite{GKcyl} implicitly defined a symmetric function generating function for cylindric tableaux.
Define the \emph{cylindric Schur function} $s_{\lambda/\mu/d}(\bx)$ corresponding to the cylindric shape $\lambda/\mu/d$ by
\begin{equation}
s_{\lambda/\mu/d}(\bx) = \sum_{T \in \CSSYT(\lambda/\mu/d)} \bx_T,
\end{equation}
where $\bx_T = \prod_{i \geq 1} x_i^{\# i\text{'s appearing in }T}.$
Cylindric Schur functions were formally introduced by Postnikov \cite{PostnikovCylSchur} 
in connection with the quantum cohomology of the Grassmannian, and later studied by McNamara \cite{Mcnamara} and  Lee \cite{LeeCylSchur}.
The author used $P$-analogues of cylindric Schur functions to give combinatorial interpretations of certain sums of $e$-coefficients of chromatic symmetric functions of incomparability graphs of \threeone -free posets \cite{Isaiah}.
The following corollary of Theorem~\ref{thm GK det sum} expresses cylindric Schur functions as sums of determinants.
\begin{corollary}
\label{cor CSSYT det}
Let $\lambda/\mu/d$ be a cylindric shape with $c$ columns. Then
\begin{equation}
s_{\lambda/\mu/d}(\bx) = \sum_{\br \in \bbZ^c_0} \det\left[e_{\lambda_i' - \mu_j' + i - j + r_j(c+d)}(\bx) \right]_{i,j = 1}^c.
\end{equation}
\end{corollary}

A special case of cylindric Schur functions are the monomial symmetric functions corresponding to rectangular shapes. 
\begin{lemma}
For all $c,r \in \bbZ_{\geq 1}$, we have
\begin{equation}
m_{c^r}(\bx) = s_{c^r/\emptyset/0}(\bx).
\end{equation}
\end{lemma}

\subsection{Representation Theory}
In this subsection we recall some useful facts about the representation theory of the general linear and symmetric groups. Results from this subsection will be used in Section~\ref{sec SPD matrices}. Exposition of these definitions and theorems can be found in the book by Fulton and Harris \cite[Lecture 6]{FultonHarris} and in the book by Merris \cite{MerrisMultilin}. For a linear map $A: V \to V$ on a finite-dimensional vector space $V$, write $\Tr(A)$ for the trace of $A$.

\begin{defin}
For a finite-dimensional vector space $V$ and partition $\lambda$, the \emph{Schur functor} $S_\lambda$ is a functor from the category of $GL(V)$-modules to itself. For each $\lambda$, $S_\lambda(V)$ is an irreducible $GL(V)$-module. For a linear map $A: U \to W$ between finite-dimensional vector spaces, $S_\lambda(A)$ is a linear map from $S_\lambda(U)$ to $S_\lambda(W)$.

For a partition $\lambda \vdash n$, write $\Sp_\lambda$ for the \emph{Specht module} associated to $\lambda$. Specht modules are the irreducible $S_n$-modules. Write $\chi^\lambda$ for the character of $\Sp_\lambda$, and write $f^\lambda$ for the dimension of $\Sp_\lambda$. Equivalently, $f^\lambda$ is the number of standard Young tableaux of shape $\lambda$.
\end{defin}

\begin{theorem}[Schur--Weyl Duality]
Let $V$ be a finite-dimensional vector space. The symmetric group $S_n$ acts on $V^{\otimes n}$ by $\sigma \cdot \bv_1 \otimes \bv_2 \otimes \dots \otimes \bv_n = \bv_{\sigma^{-1}(1)} \otimes \bv_{\sigma^{-1}(2)} \otimes \dots \otimes \bv_{\sigma^{-1}(n)}$, and this action commutes with the natural action of $GL(V)$. As a $GL(V) \times S_n$-module, 
\begin{equation}
V^{\otimes n} \cong \bigoplus_{\lambda \vdash n} S_\lambda(V) \otimes \Sp_\lambda.
\end{equation}
\end{theorem}

\begin{theorem}
Let $A: V \to V$ be a linear map. Then 
\begin{equation}
\Tr(S_\lambda(A)) = s_\lambda[A],
\end{equation}
where $s_\lambda[A]$ is the Schur function $s_\lambda$ evaluated at the eigenvalues of $A$.
\end{theorem}

\begin{defin}
For an $n \times n$ matrix $A$ and an $S_n$-class function $\psi$, the \emph{$\psi$-immanant} $\imm_\psi(A)$ of $A$ is given by
\begin{equation}
\imm_\psi(A) = \sum_{\pi \in S_n} \psi(\pi) \prod_{i = 1}^n A_{i, \pi(i)}.
\end{equation}
Write $\imm_\lambda(A)$ for the immanant $\imm_{\chi^\lambda}(A)$ corresponding to the irreducible character $\chi^\lambda$.
\end{defin}

\begin{theorem}\cite[Schur's Inequality 7.6]{MerrisMultilin}
Let $\lambda \vdash n$ and let $A$ be an $n \times n$ Hermitian positive semidefinite matrix. Then
\begin{equation}
\imm_\lambda(A) \geq f^\lambda \det(A) \geq 0.
\end{equation} 
\end{theorem}

\begin{defin}
\label{def schur proj}
For a partition $\lambda \vdash n$, define the \emph{symmetrizer} $P_\lambda: V^{\otimes n} \to V^{\otimes n}$ by
\begin{equation}
P_\lambda = \frac{f^\lambda}{n!} \sum_{\pi \in S_n} \chi^\lambda(\pi) \pi.
\end{equation}
\end{defin}
\begin{theorem}
For each $\lambda \vdash n$, the image of $P_\lambda$ is a the orthogonal projection from $V^{\otimes n}$ onto the isotypic component $S_\lambda(V) \otimes \Sp_\lambda$. Hence, $P_\lambda^* = P_\lambda$ and $P_\lambda^2 = P_\lambda$.
\end{theorem}

\begin{theorem}
\label{thm norm imm}
\cite[Corollary 7.27]{MerrisMultilin}
For $\bv = \bv_1 \otimes \bv_2 \otimes \dots \otimes \bv_n \in V^{\otimes n},$ let $K = \left[\bv_i^* \bv_j\right]_{i,j = 1}^n$ be the \emph{Gram matrix} corresponding to $\bv_1,\bv_2,\dots, \bv_n$. Then
\begin{equation}
\imm_\lambda(K) = \Tr(P_\lambda \bv \bv^*) = \|P_\lambda \bv\|^2 \geq 0.
\end{equation}
\end{theorem}

\section{Symmetric Function Evaluation Inequalities}
\label{sec sym evals}
In this section, we prove Theorem~\ref{thm sym eval ineq}. In particular, for $\bu,\bv \in \mathbb{R}_{\geq 0}^n$, $s_\lambda(\bu) \geq s_{\lambda}(\bv)$ for all partitions $\lambda$ if and only if $\bu$ log-majorizes $\bv$. 
For a polynomial $F(t) = \prod_{i=1}^n (u_i t + 1)$, we call $\bu = (u_1,u_2,\dots, u_n)$ the \emph{inverted roots} of $F(t)$.
As a consequence of Theorem~\ref{thm sym eval ineq} and the Schur-positivity of the chromatic symmetric functions of incomparability graphs of \threeone -free posets, we have that if $H$ is an induced subgraph of a claw-free incomparability graph $G$, then the inverted roots of the independence polynomial $I_G(t)$ log-majorize the inverted roots of $I_H(t)$.
Additionally, we provide necessary conditions for Schur and $e$-positivity of $\Omega_F(\bz, \bx)$ for $F(\bz, t) \in \mathbb{R}[\bz, t]$ in terms of symmetric function evaluations of inverted roots of $F(\br, t)$ for $\br \in \mathbb{R}_{\geq 0}^n$. 
For a partition $\lambda$ and vector $\bu \in \mathbb{R}^n$, let $\bu^\lambda = (\bu^\alpha \mid \alpha \in \mathbb{Z}_{\geq 0}^n, \sort(\alpha) = \lambda)$.
Recall the statement of Theorem~\ref{thm sym eval ineq}.
\begin{theorem}[Theorem~\ref{thm sym eval ineq}]
\label{thm sym eval ineq proof}
Let $\bu, \bv \in \mathbb{R}^n_{\geq 0}$.
The following are equivalent.
\begin{enumerate}
\item $\bu \succeq_{\log} \bv$,
\item $m_\lambda(\bu) \geq m_\lambda(\bv)$ for all partitions $\lambda$,
\item $s_\lambda(\bu) \geq s_\lambda(\bv)$ for all partitions $\lambda$,
\item $\Omega(\bu, \bx) - \Omega(\bv, \bx)$ is Schur-positive,
\item $\Omega(\bu, \bx) - \Omega(\bv, \bx)$ is $e$-positive.
\end{enumerate}
\end{theorem}
\begin{lemma}
For a partition $\lambda$ and vectors $\bu, \bv \in \mathbb{R}^n_{\geq 0}$, if $\bu \preceq_{\log} \bv$, then $\bu^\lambda \preceq_{\log} \bv^\lambda$. 
\end{lemma}
\begin{proof}
Let $\lambda$ be a partition with $n$ parts, and let $S_n\lambda = \{\alpha \vDash |\lambda| \mid \sort(\alpha) = \lambda\}$. Consider the linear map $T_\lambda: \mathbb{R}^n \to \mathbb{R}^{S_n\lambda}$ defined by $T_\lambda: \bu \mapsto (\alpha^T \bu \mid \alpha \in \mathbb{Z}^n_{\geq 0}, \, \sort(\alpha) = \lambda)$. 
Observe that for any $\pi \in S_n$, we have $\alpha^T \pi(\bx) = \pi^{-1}(\alpha)^T \bx$. Furthermore, the map $\alpha \mapsto \pi^{-1}(\alpha)$ is a permutation of $S_n\lambda$.
Thus, for any permutation map $P: \mathbb{R}^n \to \mathbb{R}^n$, there is a permutation map $Q: \mathbb{R}^{S_n \lambda} \to \mathbb{R}^{S_n\lambda}$ such that $T_\lambda P = Q T_\lambda$. 
Hence, by Lemma~\ref{lem DV linear pres}, if $\bu \preceq \bv$, then $T_\lambda\bu \preceq T_\lambda\bv$. 
As $T_\lambda$ has only non-negative entries, if $\bu \leq_{entry} \bv$, then $T_\lambda \bu \leq_{entry} T_\lambda\bv$.
Observe that for $\bu \in \mathbb{R}^n_{\geq 0}$, $T_\lambda(\log \bu) = \log (\bu^\lambda)$, and recall that $\bu \preceq_{\log} \bv$ is equivalent to $\log \bu \preceq_w \log \bv$. 
By Lemma~\ref{lem Ando weak maj}, if $\log \bu \preceq_w \log \bv$, then there exists $\bw \in \mathbb{R}^n$ such that $\log \bu \leq_{entry} \bw \preceq \log \bv$. Therefore,
\begin{equation}
\log(\bu^\lambda) = T_\lambda(\log \bu) \leq_{entry} T_\lambda(\bw) \preceq T_\lambda(\log \bv) = \log(\bv^\lambda).
\end{equation}
Thus, by Lemma~\ref{lem Ando weak maj}, $\log(\bu^\lambda) \preceq_w \log(\bv^\lambda)$ and $\bu^\lambda \preceq_{\log} \bv^\lambda$, as desired.
\end{proof}
\begin{corollary}
\label{cor maj imples mon}
For a partition $\lambda$ and vectors $\bu, \bv \in \mathbb{R}^n_{\geq 0}$, if $\bu \preceq_{\log} \bv$, then $\bu^\lambda \preceq_w \bv^\lambda$. In particular,
\begin{equation}
m_\lambda(\bu) = \sum_{\alpha:\, \sort(\alpha) = \lambda} \bu^\alpha \leq \sum_{\alpha:\, \sort(\alpha) = \lambda} \bv^\alpha = m_\lambda(\bv).
\end{equation}
\end{corollary}

\begin{lemma}
\label{lem Schur implies maj}
Let $\bu, \bv \in \mathbb{R}^n_{\geq 0}$. If $s_{\lambda}(\bu) \leq s_\lambda(\bv)$ for all partitions $\lambda$, then $\bu \preceq_{\log} \bv$. 
\end{lemma}
\begin{proof}
Suppose $\bu \not\preceq_{\log} \bv.$ Then for some $k$, 
\begin{equation}
\prod_{i=1}^k u_{[i]} > \prod_{i=1}^k v_{[i]}.
\end{equation}
Consider $s_{k^m}(\bu) = \sum_{T \in \SSYT(k^m, n)} \bu^T$. By Corollary~\ref{cor ssyt pol}, $|\SSYT(k^m, n)|$ is a polynomial in $m$. Hence, we have that $s_{k^m}(\bu)$ behaves like $(u_{[1]}u_{[2]}\cdots u_{[k]})^m$ as $m$ tends to infinity. Therefore, if $u_{[1]}u_{[2]}\cdots u_{[k]} > v_{[1]}v_{[2]}\cdots v_{[k]}$, 
then $s_{k^m}(\bu) > s_{k^m}(\bv)$ for some sufficiently large $m$. Thus, if $s_{\lambda}(\bu) \leq s_{\lambda}(\bv)$ for all partitions $\lambda$, then $\bu \preceq_{\log} \bv$.
\end{proof}

\begin{proof}[Proof of Theorem~\ref{thm sym eval ineq proof}]
Let $\bu, \bv \in \mathbb{R}^n_{\geq 0}$. If $\bu \preceq_{\log} \bv$, then $m_\lambda(\bu) \leq m_\lambda(\bv)$ for all partitions $\lambda$ by Corollary~\ref{cor maj imples mon}. As Schur functions are monomial-positive, if $m_\lambda(\bu) \leq m_\lambda(\bv)$ for all partitions $\lambda$, then $s_\lambda(\bu) \leq s_\lambda(\bv)$ for all partitions $\lambda$. By Lemma~\ref{lem Schur implies maj}, if $s_\lambda(\bu) \leq s_\lambda(\bv)$ for all partitions $\lambda$, then $\bu \preceq_{\log} \bv$. Hence, conditions (1), (2), and (3) are equivalent. As $\Omega(\bu, \bx) = \sum_{\lambda} s_\lambda(\bu) s_{\lambda'}(\bx) = \sum_{\lambda} m_\lambda(\bu) e_\lambda(\bx)$,
\begin{equation}
\Omega(\bu, \bx) - \Omega(\bv, \bx) = \sum_{\lambda} \left[s_\lambda(\bu) - s_\lambda(\bv)\right] s_{\lambda'}(\bx) = \sum_{\lambda} \left[m_\lambda(\bu) - m_\lambda(\bv)\right] e_\lambda(\bx).
\end{equation}
Therefore, condition (3) is equivalent to condition (4), and condition (2) is equivalent to condition (5). 
\end{proof}

\begin{corollary}
Let $F(\bz, t) \in \mathbb{Q}[z_1,z_2,...,z_n,t]$ and suppose $\Omega_F(\bz, \bx)$ is Schur-positive. Let $\br, \bs \in \mathbb{R}_{\geq 0}^n$ with $\br \geq_{entry} \bs$, let $\bu$ be the inverted roots of $F(\br, t)$, and let $\bv$ be the inverted roots of $F(\bs, t)$. Then $\bu \succeq_{\log} \bv$.
\end{corollary}
\begin{proof}
As $\Omega_F(\bz, \bx)$ is Schur-positive, $s_\lambda(\bu) = s_\lambda^F(\br) \geq s_\lambda^F(\bs) = s_\lambda(\bv)$ for all partitions $\lambda$. Hence, $\bu \succeq_{\log} \bv$. 
\end{proof}

As a corollary of Theorem~\ref{thm Gasharov} and Theorem~\ref{thm sym eval ineq proof}, we obtain the following inequalities of the inverted roots of independence polynomials of incomparability graphs of \threeone -free posets. 

\begin{corollary}
\label{cor 3+1 inv root log maj}
Let $G$ be the incomparability graph of a \threeone -free poset on $[n]$, let $\br \geq_{entry} \bs \in \mathbb{R}^n_{\geq 0}$, let $\bu$ be the inverted roots of $I_G(\br, t)$, and let $\bv$ be the inverted roots of $I_G(\bs, t)$. Then $\bu \succeq_{\log} \bv$.
\end{corollary}

%

We obtain stronger inequalities for symmetric function evaluations from Schur and $e$-positivity of $\Omega_F(\bz, \bx)$. For $S \subseteq [n]$, write $\mathbf{1}_S \in \mathbb{R}^n$ for the indicator vector of $S$.

\begin{theorem}
\label{thm IncExc ineq}
Let $F(\bz, t) \in \mathbb{Q}[z_1,z_2,\dots, z_n, t]$.
For $S \subseteq [n]$, let $\bu_S \in \mathbb{R}^n$ be the inverted roots of $F(\mathbf{1}_S, t)$. 
If $\Omega_F(\bz, \bx)$ is Schur-positive, then
\begin{equation}
\sum_{S \subseteq [n]} (-1)^{n - |S|} s_\lambda(\bu_S) \geq 0,
\end{equation}
for all partitions $\lambda$.
Likewise, if $\Omega_F(\bz, \bx)$ is $e$-positive, then
\begin{equation}
\sum_{S \subseteq [n]} (-1)^{n - |S|} m_\lambda(\bu_S) \geq 0,
\end{equation}
for all partitions $\lambda$.
\end{theorem}
\begin{proof}
Write $\Omega_F(\bz, \bx) = \sum_{\alpha \in \bbZ^n_{\geq 0}} \bz^\alpha \Omega_F^\alpha(\bx)$.
From a straightforward inclusion-exclusion argument, we have
\begin{equation}
\sum_{S \subseteq [n]} (-1)^{n - |S|} \Omega_F(\mathbf{1}_S, \bx) = \sum_{\alpha \in \bbZ^{n}_{\geq 1}} \Omega_F^{\alpha}(\bx).
\end{equation}
Thus, if $\Omega_F(\bz, \bx)$ is Schur-positive ($e$-positive), then $\sum_{S \subseteq [n]} (-1)^{n - |S|} \Omega_F(\mathbf{1}_S, \bx)$ is Schur-positive ($e$-positive).
As $\Omega_F(\bz, \bx) = \sum_{\lambda} s_\lambda^F(\bz) s_{\lambda'}(\bx) = \sum_{\lambda} m_\lambda^F(\bz) e_\lambda(\bx)$, $s_\lambda^F(\mathbf{1}_S) = s_\lambda(\bu_S)$, and $m_\lambda^F(\mathbf{1}_S) = m_\lambda(\bu_S)$, we have 
\begin{align}
\sum_{S \subseteq [n]} (-1)^{n - |S|} \Omega_F(\mathbf{1}_S, \bx) &= \sum_\lambda \sum_{S \subseteq [n]} (-1)^{n-|S|} s_\lambda(\bu_S) s_{\lambda'}(\bx) \\
&= \sum_\lambda \sum_{S \subseteq [n]} (-1)^{n - |S|} m_\lambda(\bu_S) e_\lambda(\bx).
\end{align}
Hence, $\sum_{S \subseteq [n]} (-1)^{n - |S|}s_\lambda(\bu_S)$ is the coefficient of $s_{\lambda'}(\bx)$ and $\sum_{S \subseteq [n]} (-1)^{n-|S|}m_\lambda(\bu_S)$ is the coefficient of $e_\lambda(\bx)$ in $\sum_{S \subseteq [n]} (-1)^{n - |S|} \Omega_F(\mathbf{1}_S, \bx)$, and the result follows.
\end{proof}

\begin{corollary}
Let $G = \inc(P)$ for a \threeone -free poset $P$ on $[n]$. For $S \subseteq [n]$, let $\bu_S \in \bbR_{\geq 0}^d$ be the inverted roots of $I_G(\mathbf{1}_S, t)$. Then for all partitions $\lambda$,
\begin{equation}
\sum_{S \subseteq [n]} (-1)^{n - |S|} s_\lambda(\bu_S) \geq 0,
\end{equation}
and
\begin{equation}
\sum_{S \subseteq [n]} (-1)^{n-|S|} m_\lambda(\bu_S) \geq 0.
\end{equation}
\end{corollary}


\section{Schur Positivity From Networks}
\label{sec s pos from networks}
In this section, we use the results of Subsection~\ref{sec LGV background} to prove Schur-positivity results for chromatic symmetric functions. In particular, we prove Theorem~\ref{thm s pos XincP Gamma h} and Theorem~\ref{thm s pos matching sum}.
Both theorems are proven using networks $\Gamma_\bh$ indexed by Hessenberg functions. 


\begin{defin}
Let $\bh = (h_1, h_2,...,h_n) \in \mathbb{H}_n$ be a Hessenberg function. Define the planar network $\Gamma_{\bh}$ with weights in $\bbQ[z_1,z_2,\dots,z_n]$ to be the directed graph on vertices $V_\bh = \bbZ \times [n+1]$ with edges of weight $1$ from $(i,j)$ to $(i, j+1)$ and edges of weight $z_{j}$ from $(i,j)$ to $(i+1, h_{j} + 1)$. Write $s_i^\bh = (i, 1)$ for the source vertices and $t_i^\bh = (i, n+1)$ for the sink vertices of $\Gamma_\bh$.
 As $h_i \leq h_{i+1}$ for each $i$, we have that $\Gamma_\bh$ is planar, so $\Gamma_\bh$ is compatible with the source and sink labels $\bs^\bh$ and $\bt^\bh$. 
\end{defin}

\begin{figure}
\caption{$\Gamma_{\bh}$ for $\bh = (2,3,4,4)$.}
\[\begin{tikzcd}[ampersand replacement=\&]
	\& \bullet \& \bullet \& \bullet \& \bullet \& {} \\
	\dots \& \bullet \& \bullet \& \bullet \& \bullet \& \dots \\
	\dots \& \bullet \& \bullet \& \bullet \& \bullet \& \dots \\
	\dots \& \bullet \& \bullet \& \bullet \& \bullet \& \dots \\
	{} \& \bullet \& \bullet \& \bullet \& \bullet
	\arrow[from=2-1, to=1-2]
	\arrow[from=2-2, to=1-2]
	\arrow["{z_4}"{description}, from=2-2, to=1-3]
	\arrow[from=2-3, to=1-3]
	\arrow["{z_4}"{description}, from=2-3, to=1-4]
	\arrow[from=2-4, to=1-4]
	\arrow["{z_4}"{description}, from=2-4, to=1-5]
	\arrow[from=2-5, to=1-5]
	\arrow[from=2-5, to=1-6]
	\arrow[from=3-1, to=1-2]
	\arrow["{z_3}"{description}, from=3-2, to=1-3]
	\arrow[from=3-2, to=2-2]
	\arrow["{z_3}"{description}, from=3-3, to=1-4]
	\arrow[from=3-3, to=2-3]
	\arrow["{z_3}"{description}, from=3-4, to=1-5]
	\arrow[from=3-4, to=2-4]
	\arrow[from=3-5, to=1-6]
	\arrow[from=3-5, to=2-5]
	\arrow[from=4-1, to=2-2]
	\arrow["{z_2}"{description}, from=4-2, to=2-3]
	\arrow[from=4-2, to=3-2]
	\arrow["{z_2}"{description}, from=4-3, to=2-4]
	\arrow[from=4-3, to=3-3]
	\arrow["{z_2}"{description}, from=4-4, to=2-5]
	\arrow[from=4-4, to=3-4]
	\arrow[from=4-5, to=2-6]
	\arrow[from=4-5, to=3-5]
	\arrow[from=5-1, to=3-2]
	\arrow["{z_1}"{description}, from=5-2, to=3-3]
	\arrow[from=5-2, to=4-2]
	\arrow["{z_1}"{description}, from=5-3, to=3-4]
	\arrow[from=5-3, to=4-3]
	\arrow["{z_1}"{description}, from=5-4, to=3-5]
	\arrow[from=5-4, to=4-4]
	\arrow[from=5-5, to=3-6]
	\arrow[from=5-5, to=4-5]
\end{tikzcd}\]
\end{figure}

We prove Theorem~\ref{thm s pos XincP Gamma h}. By Theorem~\ref{thm Gasharov}, the coefficient of $s_\lambda(\bx)$ in $X_{\inc(P_\bh)}(\bx)$ is given by the number of $P_\bh$-tableaux of shape $\lambda'$. Hence, Theorem~\ref{thm XG s pos Gamma proof} below is a restatement of Theorem~\ref{thm s pos XincP Gamma h}.

\begin{theorem}
\label{thm XG s pos Gamma proof}
Let $\bh = (h_1,h_2,\dots, h_n) \in \mathbb{H}_n$ be a Hessenberg function. Then for all partitions $\lambda$, there is a bijection between non-intersecting path families from $(s_1^\bh, s_2^\bh, \dots, s_\ell^\bh)$ to $(t_{\lambda_\ell}^\bh, t_{\lambda_{\ell - 1} + 1}^\bh,\dots, t_{\lambda_1 + \ell - 1}^\bh)$ and $P_\bh$-tableaux of shape $\lambda'$.
\end{theorem} 
\begin{proof}
Let $\gamma$ be a path from $s_i^\bh$ to $t_j^\bh$. Observe that $\wt(\gamma)$ will be $z_{i_1}z_{i_2} \dots z_{i_{j-i}}$ where $i_k <_{P_\bh} i_{k+1}$ for all $1 \leq k < j-i$ and that $\gamma$ is the unique path from $s_i^\bh$ to $t_j^\bh$ with this weight. 
Hence, $\wt(s_i^\bh, t_j^\bh) = e_{j-i}^{\inc(P_\bh)}(\bz)$. Therefore, $W_\Gamma = \left[e_{j-i}^{\inc(P_\bh)}(\bz)\right]_{i,j=1}^\infty.$ 
From Theorem~\ref{thm Lindstrom}, this shows that $s_{\lambda}^{\inc(P_\bh)}(\bz)$ is in $\bbZ_{\geq 0}[\bz]$ for all partitions $\lambda$. Moreover, we have
\begin{equation}
s_{\lambda'}^{\inc(P_\bh)}(\bz) = \sum_{A \in \bN(\ba,\bb)} \wt(A),
\end{equation}
where $\ba = (1,2,\dots, \ell)$ and $\bb = (\lambda_\ell, \lambda_{\ell-1} + 1, \dots, \lambda_1 + \ell - 1)$.
Let $(\gamma_1, \gamma_2, \dots, \gamma_\ell) \in \bP(\ba, \bb)$. 
Then $\wt(\gamma_i) = z_{c_{i,1}} z_{c_{i,2}} \cdots z_{c_{i,\lambda_{\ell - i+1}}}$ for some $c_{i,1} <_P c_{i,2} <_P \dots <_P c_{i, \lambda_{\ell - i+1}}$. Observe that $\gamma_i$ will intersect $\gamma_{i+1}$ if and only if there is some $r$ such that $c_{i, r} <_P c_{i+1, r}$. Hence, by Lemma~\ref{lem adj path intersect}, $(\gamma_1, \gamma_2,\dots, \gamma_\ell) \in \bN(\ba,\bb)$ if and only if the filling $T$ of $\lambda'$ given by $T_{i,j} = c_{j, \ell - i + 1}$ is a $P_\bh$-tableau. 
\end{proof}

\begin{remark}
There is no obvious way to extend Theorem~\ref{thm XG s pos Gamma proof} to \threeone -free posets that are not natural unit interval orders. In particular, the author was unable to find a network with polynomial weights that realizes Schur-positivity for the poset $\twotwo$.
\end{remark}

The remainder of this section focuses on Schur-positive sums of chromatic symmetric functions. In particular, we prove Theorem~\ref{thm s pos matching sum}. The main tool will be the \emph{$r^{th}$ reweighting} $\Gamma^{(r)}$ of a network $\Gamma$.

\begin{defin}
Let $\mathcal{R} = \mathbb{Q}[z_1, z_2,\dots, z_n]$, and let $\mathsf{S}: \Gamma \to \Gamma$ be a shift automorphism. For a positive integer $k$, the \emph{$r^{th}$ reweighting} $\Gamma^{(r)}$ of $\Gamma$ is a network with weights in $\mathbb{Q}[z_{i, j} \mid 1 \leq i \leq n, 1 \leq j \leq r]$ defined as follows.
\begin{enumerate}
\item As unweighted directed networks, $\Gamma^{(r)}$ is isomorphic to $\Gamma$.

\item If $e$ is an edge in $\Gamma$ with weight $\bz^{\alpha}$ which appears in a path from $s_0$ to $t_1$, the corresponding edge in $\Gamma^{(r)}$ has weight $z_{1,1}^{\alpha_1} z_{2,1}^{\alpha_2} \cdots z_{n,1}^{\alpha_n}$.

\item If $e$ is an edge in $\Gamma^{(r)}$ with weight $z_{1,i}^{\alpha_1} z_{2, i}^{\alpha_2} \cdots z_{n, i}^{\alpha_n}$, then $\mathsf{S}(e)$ has weight $z_{1, i+1}^{\alpha_1} z_{2, i+1}^{\alpha_2} \cdots z_{n, i+1}^{\alpha_n}$, with the convention that $z_{j, r+1} = z_{j, 1}$.

\item $\Gamma^{(r)}$ has sources $s_i^{\Gamma^{(r)}} = s_{ik}^{\Gamma}$ and sinks $t_{i}^{\Gamma^{(r)}} = t_{ir}^{\Gamma}$.
\end{enumerate}
\end{defin}

\begin{defin}
For a poset $P$ on $[n]$, write $\chain_r(P)$ for the set of $r$-element chains in $P$. For a multiset $M$ of elements of $\chain_r(P)$, let $\deg(M) \in \bbZ_{\geq 0}^{n \times r}$ be the matrix such that $\deg(M)_{i,j}$ is the number of times $i$ appears as the $r^{th}$ largest entry of an element of $M$. For $\bd \in \bbZ_{\geq 0}^{n \times r}$, write $\mathcal{M}(\bd, P)$ for the set of multisets $M$ with elements in $\chain_r(P)$ such that $\deg(M) = \bd$. Observe that $\mathcal{M}(\bd, P)$ will be empty unless each column of $\bd$ sums to the same number.
\end{defin}

Recall that for a collection $M$ of real numbers, we write $P_M$ for the interval poset associated with $M$. 
For a multiset $M$ of elements in ${\bbR \choose r}$, let $P_M$ be the interval poset obtained by identifying each $r$-element with the interval between its minimum and maximum elements. The following theorem is a generalization of Theorem~\ref{thm s pos matching sum}. Theorem~\ref{thm s pos matching sum} follows as a special case by taking $r = 2$ and $\bh = (1,2,\dots, n)$. Details of the proof are postponed to later in the section. 
In general, the individual terms of the sum will not be Schur-positive.

\begin{theorem}
\label{thm s pos sum proof}
Let $\bh \in \mathbb{H}_n$ be a Hessenberg function, and let $P_\bh$ be the associated natural unit interval order. Then for $\bd \in \bbZ_{\geq 0}^{n \times r}$, the symmetric function
\begin{equation}
Y_{\bd, P_\bh}(\bx) := \sum_{M \in \mathcal{M}(\bd, P_\bh)} \displaystyle{\frac{ X_{\inc(P_M)}(\bx)}{\prod_{e \in \chain_r(P)} m_e(M)!} }
\end{equation}
is Schur-positive.
\end{theorem}
\begin{proof}
Let $e_k^{\bh, (r)} = \sum_{\bi \in \chain_{kr}(P_\bh)} \prod_{j = 1}^{kr} z_{i_j, j'}$, where $1 \leq j' \leq r$ is the index such that $j' \equiv j \mod r$. Let $F_{\bh, (r)}(t) = \sum_{k \geq 0} e_k^{\bh,(r)}(\bz) t^k$.
From Lemma~\ref{lem Omega Y exp}, 
\begin{equation}
\Omega_{F_{\bh, (r)}}(\bz, \bx) = \sum_{\bd \in \bbZ_{\geq 0}^{n \times r}} \left( \prod_{i=1}^n \prod_{j = 1}^r z_{i,j}^{d_{i,j}}\right) Y_{\bd, P_\bh}(\bx).
\end{equation}
Observe that the weight matrix of $\Gamma_\bh^{(r)}$ is given by $W_{\Gamma_\bh^{(r)}} = \left[e_{j-i}^{\bh, (r)}\right]_{i,j=1}^{\infty}$. Hence, $\Omega_{F_{\bh, (r)}}(\bz, \bx)$ is Schur-positive. As $Y_{\bd, P_\bh}(\bx)$ is the coefficient of $\prod_{i=1}^n \prod_{j = 1}^r z_{i,j}^{d_{i,j}}$ in $\Omega_{F_{\bh, (r)}}(\bz, \bx)$, we have that $Y_{\bd, P_\bh}(\bx)$ is Schur-positive.
\end{proof}

\begin{exmp}
Let $r = 2$ and $\bh = (1,2,\dots, 6) \in \mathbb{H}_n$. In this case, $P_\bh$ is the total order on $[6]$, so $\chain_r(P_\bh)$ is the set of all pairs $\{i < j\} \subset [6]$. Thus, mutlisets of elements of $\chain_r(P_\bh)$ are multigraphs on $[6]$. For $\bd$ with columns $(2, 0, 1, 0, 1, 0)$ and $(0, 1, 0, 1, 0, 2),$ $\mathcal{M}(\bd, P_\bh)$ consists of the sets $M_1 = \{12, 16, 34, 56\}$ and $M_2 = \{12, 14, 26, 56\}$.
The incomparability graphs of the associated interval posets are the following.
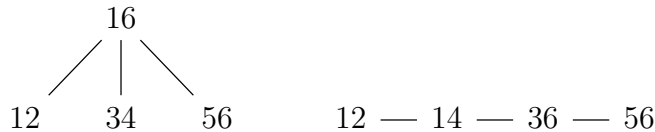
\[\begin{tikzcd}[ampersand replacement=\&,column sep=small]
	\& 16 \&\&\&\&\&\& \\
	12 \& 34 \& 56 \&\& 12 \& 14 \& 36 \& 56
	\arrow[no head, from=1-2, to=2-2]
	\arrow[no head, from=1-2, to=2-3]
	\arrow[no head, from=2-1, to=1-2]
	\arrow[no head, from=2-5, to=2-6]
	\arrow[no head, from=2-6, to=2-7]
	\arrow[no head, from=2-7, to=2-8]
\end{tikzcd}\]
As a consequence of Theorem~\ref{thm s pos XincP Gamma h}, we have that $X_{\inc(P_{M_1})}(\bx) + X_{\inc(P_{M_2})}(\bx)$ is Schur-positive. Indeed,
\begin{align}
X_{\inc(P_{M_1})}(\bx) + X_{\inc(P_{M_2})}(\bx) &= \left(s_{(3,1)}(\bx) - s_{(2,2)}(\bx) + 5 s_{(2,1,1)}(\bx) + 8 s_{(1,1,1,1)}\right)\\
& + \left(2s_{(2,2)}(\bx) + 4 s_{(2, 1, 1)}(\bx) + 8 s_{(1,1,1,1)}(\bx) \right) \\
&= s_{(3,1)}(\bx) + s_{(2,2)}(\bx) + 9 s_{(2,1,1)}(\bx) + 16 s_{(1,1,1,1)}(\bx).
\end{align}
\end{exmp}

The symmetric function $PM_{K_{2n}}(\bx)$ defined in Section~\ref{sec chromaticbackground} can be expressed as a sum of symmetric functions $Y_{\bd, P_\bh}(\bx)$. Take $Z_{2n} \subset \bbZ_{\geq 0}^{2n \times 2}$ to be the set of $\{0,1\}$-matrices such that each row has exactly one $1$ and each column has exactly $n$ 1's.  For $\bd \in Z_{2n}$, we have that $\mathcal{M}(\bd, P_\bh)$ is a set of perfect matchings in the comparability graph of $P_\bh$.
Taking $\bh = (1,2,\dots, 2n) \in \mathbb{H}_{2n}$, we have that
\begin{equation}
 PM_{K_{2n}}(\bx) = \sum_{\bd \in Z_{2n}} Y_{\bd, P_\bh}.
\end{equation}
In Theorem~\ref{thm AmdShaSta match}, Amdeberhan--Shareshian--Stanley \cite{AmdShaSta} showed that $PM_{K_{2n}}(\bx)$ is $e$-positive for all $n$. 
We conjecture the following refinement of Theorem~\ref{thm AmdShaSta match}. Conjecture~\ref{conj Yd e pos} has been checked for $r = 2$ for all $\bd$ with column sums at most 6. 
\begin{conjecture}
\label{conj Yd e pos}
Let $\bh \in \mathbb{H}_n$ and let $\mathbb{Z}_{\geq 0}^{n \times r}$. Then $Y_{\bd, P_\bh}(\bx)$ is $e$-positive.
\end{conjecture}

Conjecture~\ref{conj Yd e pos} is closely related to Conjecture~\ref{conj AlexanderssonDai} of Alexandersson--Dai \cite{AlexanderssonDai}. Recall that $T_k: \Lambda \to \Lambda$ is the map which sends $e_i(\bx)$ to $e_{ik}(\bx)$. Conjecture~\ref{conj AlexanderssonDai} states that $T_k(m_\lambda(\bx))$ is monomial-positive for all partitions $\lambda$. An equivalent statement to  Conjecture~\ref{conj Yd e pos} is that $m_\lambda^{\bh, (r)}(\bz)$ is in $\mathbb{Q}_{\geq 0}[z_{i,j}]_{i \in [n], j \in [r]}$ for all partitions $\lambda$. Taking $\bh = (1,2,\dots, n) \in \mathbb{H}_n$, and sending $z_{i,1}, z_{i,2}, \dots, z_{i,r}$ to $z_i$ gives $T_r(m_\lambda(\bz))$. Hence, Conjecture~\ref{conj Yd e pos} implies Conjecture~\ref{conj AlexanderssonDai}. The following result shows elementary symmetric functions indexed by rectangular-shape partitions have non-negative coefficients in $\Omega_{F^\Gamma}(\bz, \bx)$ when $\Gamma$ has a nice shift automorphism.
As a corollary, we have that the coefficient of $e_{c^r}(\bx)$ in $Y_{\bd, P_\bh}(\bx)$ is non-negative.

\begin{theorem}
\label{thm Gamma rect e pos}
Let $\Gamma$ be a network with weights in $\mathcal{R}_+$ and a shift automorphism $\mathsf{S}$. If there is no integer $m$ and vertex $v$ of $\Gamma$ such that there is a path from $v$ to $\mathsf{S}^m(v)$, then $\omega \Omega_\Gamma(\bx)$ has non-negative coefficients in the powersum basis. Moreover, the coefficient of $e_{c^r}(\bx)$  in $\Omega_\Gamma(\bx)$ is non-negative for all $r,c \geq 0$.
\end{theorem}

\begin{proof}
Write $e_{j-i}^\Gamma = \wt(s_i^\Gamma, t_j^\Gamma)$. As $\Gamma$ is shift-symmetric, $e_k^\Gamma$ is well defined for all $k$. For a symmetric function $g(\bx) = \sum_{\lambda} a_\lambda e_\lambda(\bx)$, define $g^\Gamma = \sum_{\lambda} a_\lambda e_\lambda^\Gamma$.
From Theorem~\ref{thm GK det sum}, we have that cylindrically nonintersecting paths of period $d$ from $\ba = (a_1,\dots, a_k)$ to $\bb = (b_1, \dots, b_k)$ in $\Gamma$ are given by 
\begin{equation}
\sum_{A \in \bCN(\ba, \bb, d)} \wt(A) = \sum_{\br \in \bbZ_{0}^k} \det\left[\wt(s_{a_\alpha}, t_{b_\beta + r_\alpha d)}\right]_{\alpha, \beta = 1}^k = \sum_{\br \in \bbZ_{0}^k} \left[e_{b_\beta - a_\alpha + r_\alpha d}^\Gamma(\bz)\right]_{\alpha, \beta = 1}^d.
\end{equation}
For a partition $\lambda$ of length $\ell$, let $\ba_\lambda = (1,2,\dots, \ell)$ and let $\bb_\lambda = (\lambda_\ell, \lambda_{\ell-1} + 1, \dots, \lambda_1 + \ell - 1)$. Let $d > \lambda_1 - \lambda_\ell$. Then by Corollary~\ref{cor CSSYT det},
\begin{equation}
s_{\lambda'/\emptyset/d}^\Gamma = \sum_{\br \in \bbZ_0^{k}} \det\left[e_{\lambda_i - i + j + r_i(\ell + d)}^\Gamma\right]_{i,j = 1}^\ell = \sum_{A \in \bCN(\ba_\lambda, \bb_\lambda, d + \ell)} \wt(A).
\end{equation}
Thus for all cylindric partitions $\lambda/\emptyset/d$, we have that $s_{\lambda/\emptyset/d}^\Gamma$ is an element of $\mathcal{R}_+$.
For the rectangular shape $c^r$, we have $m_{c^r}(\bx) = s_{c^r/\emptyset/0}(\bx)$. As the coefficient of $e_{c^r}(\bx)$ in $\Omega_\Gamma(\bx)$ is $m_{c^r}^\Gamma$, the coefficient of $e_{c^r}(\bx)$ is in $\mathcal{R}_+$. As $p_k(\bx) = m_k(\bx)$, and $m_k^\Gamma \in \mathcal{R}_+$, we have that $p_\lambda^\Gamma \in \mathcal{R}_+$ for all partitions $\lambda$. As the coefficient of $p_\lambda(\bx)$ in $\omega \Omega_\Gamma(\bx)$ is $z_\lambda^{-1} p_\lambda^\Gamma$, $\omega \Omega_\Gamma(\bx)$ has non-negative coefficients in the powersum basis.
\end{proof}

\begin{corollary}
For $\bd \in \mathbb{Z}_{\geq 0}^{n \times r}$ and $\bh \in \mathbb{H}_n$, we have that $\omega Y_{\bd, P_{\bh}}(\bx)$ has non-negative coefficients in the powersum basis, and the coefficient of $e_{c^r}(\bx)$ is non-negative in $Y_{\bd, P_\bh}(\bx)$.
\end{corollary}

\begin{proof}
As $\Omega_{\Gamma_\bh^{(r)}}(\bz, \bx) = \sum_{\bd \in \bbZ_{\geq 0}^{n \times r}} \left(\prod_{i=1}^n \prod_{j=1}^r z_{i,j}^{d_{i,j}}\right) Y_{\bd, P_{\bh}}(\bx)$, it suffices to show that $\Gamma_\bh^{(r)}$ satisfies the assumptions of Theorem~\ref{thm Gamma rect e pos}. Observe that $\Gamma_\bh^{(r)}$ has shift automorphism $\mathsf{S}$ sending vertex $(i,j)$ to $(i + r, j)$. Observe that if there is a path from $(i,j)$ to $(i', j')$ in $\Gamma_\bh^{(r)}$, then $j' > j$. Therefore, for any vertex $v = (i,j)$ and integer $m$, there is no path from $v$ to $\mathsf{S}^m(v) = (i + rm, j)$, and $\Gamma_\bh^{(r)}$ satisfies the assumptions of Theorem~\ref{thm Gamma rect e pos}.
\end{proof}

In the remainder of the section, we prove some technical lemmas which were used in the proof of Theorem~\ref{thm s pos sum proof}. In particular, we show that the coefficient of $\left(\prod_{i=1}^n \prod_{j = 1}^r z_{i,j}^{d_{i,j}}\right)$ in $\Omega_{F_{\bh, (r)}}(\bz, \bx)$ is $Y_{\bd, P_\bh}(\bx)$.

\begin{defin}
For a poset $P$ and a partition $\lambda$, a \emph{$P$-array} of shape $\lambda$ is a filling of $\lambda$ with elements of $P$ such that entries are $P$-increasing down columns. 
Write $\A_P(\lambda)$ for the set of $P$-arrays of shape $\lambda$.
A \emph{standard $P$-array} is a $P$-array in which every element of $P$ appears once.
Write $\SA_P(\lambda)$ for the set of standard $P$-arrays of shape $\lambda$.

For a partition $\lambda = (\lambda_1, \lambda_2, \dots)$ of $k$, write $\lambda^r$ for the partition of $rk$ with $r$ copies of each part $\lambda_i$. For a $P$-array $T$ of shape $\lambda^r$, let $M_T$ be the multiset $\{(T_{ri+1, j}, T_{ri+1, j},\dots, T_{r(i+1),j}) \in \chain_r(P)  \mid (ri+1, j) \in \lambda^r\}$. 
For a multiset $M$ of elements of $\chain_r(P)$, let $\A_P(\lambda^r, M) = \{T \in \A_P(\lambda^r) \mid M_T = M\}$.
Write $\wt(M) = \prod_{\bi \in M} z_{i_1, 1} z_{i_2, 2} \cdots z_{i_r,r} \in \mathbb{Q}[z_{s,t}]_{s \in [n], t\in [r]}$ for the \emph{weight of $M$}.

For a multiset $M$ of $r$-element sets of real numbers, let $\varphi_M: P_M \to \bbR^r$ be the map that sends elements of $P_M$ to the associated sets of real numbers. For a partition $\lambda \vdash k$, a poset $P$ on $[n]$, and an $k$-element multiset $M$ of chains in $\chain_r(P)$, we extend $\varphi_M$ to a map $\varphi_M: \SA_{P_M}(\lambda) \to \A_P(\lambda^r, M)$ by sending $T \in \SA_{P_M}(\lambda)$ to the $P$-array $\varphi_M(T)$ with entries
$\varphi_M(T)_{i,j} = \varphi_M(T_{i'+1, j})_{i''}$, where $i = ri' + i''$ with $1 \leq i'' \leq r$.
\end{defin}

\begin{lemma}
\label{lem inc PM array exp}
Let $P$ be a poset on $[n]$ and let $M$ be a $k$-element multiset of chains in $\chain_r(P)$. Then
\begin{equation}
\prod_{e \in \chain_r(P)} \frac{1}{m_e(M)!} X_{\inc(P_M)} = \sum_{\lambda \vdash k} |\A_P(\lambda^r, M)|\, m_{\lambda'}(\bx). 
\end{equation}
\end{lemma}

\begin{proof}
The coefficient of $m_\lambda(\bx)$ in $X_{G}(\bx)$ is the number of ways to partition the vertices of $G$ into independent sets of size $\lambda_1, \lambda_2, \dots, \lambda_l$. Hence, the coefficient of $m_\lambda(\bx)$ in $X_{\inc(P)}(\bx)$ is the number of standard $P$-arrays $|\SA_P(\lambda')|$. Hence,
\begin{equation}
X_{\inc(P_M)}(\bx) = \sum_{\lambda \vdash k} |\SA_{P_M}(\lambda)| m_{\lambda'}(\bx).
\end{equation}
Observe that for all $\lambda \vdash k$, $\varphi_M$ is a $\left(\prod_{e \in \chain_r(P)} m_e(M)!\right)$-to-one map from $\SA_{P_M}(\lambda)$ to $\A_P(\lambda^r, M)$. Hence, $|\SA_{P_M}(\lambda)| = \left(\prod_{e \in \chain_r(P)} m_e(M)!\right) |\A_P(\lambda^r, M)|$, and the result follows.
\end{proof}
\begin{lemma}
\label{lem Omega Y exp}
Let $\bh \in \mathbb{H}_n$ be a Hessenberg function, and let $P = P_\bh$ be the associated unit interval graph. Let $e_k^{\bh, (r)} = \sum_{\bi \in \chain_{kr}(P)} \prod_{j = 1}^{kr} z_{i_j, j'}$, where $1 \leq j' \leq r$ is the index such that $j' \equiv j \mod r$. Let $F_{\bh, (r)}(t) = \sum_{k \geq 0} e_k^{\bh, (r)}(\bz) t^k$. Then
\begin{equation}
\Omega_{F_{\bh, (r)}}(\bz, \bx) = \sum_{\bd \in \bbZ_{\geq 0}^{n \times r}} \left(\prod_{i = 1}^n \prod_{j = 1}^r z_{i, j}^{d_{i,j}}\right) Y_{\bd, P_\bh}(\bx).
\end{equation}
\end{lemma}
\begin{proof}
We prove the equality by comparing the coefficients of $m_{\lambda'}(\bx)$. The coefficient of $m_{\lambda'}(\bx)$ in $\Omega_{F_{\bh, (r)}}(\bz, \bx)$ is $e_{\lambda'}^{\bh, (r)}(\bz)$. 
We have
\begin{equation}
e_{\lambda'}^{\bh, (r)}(\bz) = \sum_{T \in \A_P(\lambda^r)} \wt(M_T) = \sum_{\bd \in \bbZ_{\geq 0}^{n \times r}} \sum_{M \in \mathcal{M}(\bd, P)} \wt(M) |\A_P(\lambda^r, M)|.
\end{equation}
Hence,
\begin{equation}
\Omega_{F_{\bh,(r)}}(\bz, \bx) = \sum_{\bd\in \bbZ_{\geq 0}^{n \times r}} \sum_{M \in \mathcal{M}(\bd, P)} \wt(M) \sum_{\lambda \vdash |M|} |\A_P(\lambda^r, M)| m_{\lambda'}(\bx).
\end{equation}
As $\wt(M) = \prod_{i = 1}^n \prod_{j = 1}^r z_{i,j}^{d_{i,j}}$ for all $M \in \mathcal{M}(\bd, P)$,
\begin{equation}
\Omega_{F_{\bh,(r)}}(\bz, \bx) = \sum_{\bd \in \bbZ_{\geq 0}^{n \times r}} \left(\prod_{i = 1}^n \prod_{j = 1}^r  z_{i,j}^{d_{i,j}}\right) \sum_{M \in \mathcal{M}(\bd, P)} \sum_{\lambda \vdash |M|} |\A_P(\lambda^r, M)| m_{\lambda'}(\bx).
\end{equation}
Combining Lemma~\ref{lem inc PM array exp} and the definition of $Y_{\bd,P_\bh}(\bx)$, we have
\begin{equation}
\Omega_{\bh, (r)}(\bz, \bx) = \sum_{\bd \in \bbZ_{\geq 0}^{n\times r}} \left(\prod_{i = 1}^n \prod_{j = 1}^r z_{i, j}^{d_{i,j}}\right) Y_{\bd, P_\bh}(\bx),
\end{equation}
as desired.
\end{proof}

\section{Schur-positive Symmetric Functions from Positive Semidefinite Matrices}
\label{sec SPD matrices}
In this section, we study $\Omega_{F_A}(\bz, \bx)$ for polynomials $F_A(\bz, \bx)$ of the form
\begin{equation}
F_A(\bz, t) = \det(I + t A(\bz)),
\end{equation}
where $A(\bz)$ is a matrix with nonnegative eigenvalues for all nonnegative evaluations of $\bz$.
In particular, we consider the case when $A(\bz)$ is a Hermitian matrix of the form $A(\bz) = \sum_{i = 1}^m z_i \bv_i \bv_i^*$ for a collection of vectors $\bv_1,\bv_2, \dots, \bv_n \in \mathbb{C}^n$ and prove Theorem~\ref{thm herm s pos}. Additionally, we show that if $A(\bz) = B(\bz) C(\bz)$ is a product of Hermitian matrices $B(\bz) = \sum_{i=1}^m z_i \bv_i \bv_i^*$ and $C(\bz) = \sum_{i = 1}^{m'} z_{i+m} \bw_i \bw_i^*$ for vectors $\bv_1,\bv_2,\dots,\bv_m, \bw_1,\bw_2,\dots, \bw_{m'} \in \bbC^n$, then $\Omega_{F_A}(\bz, \bx)$ is Schur-positive.
The proof of Theorem~\ref{thm herm s pos} uses basic results in the representation theory of $GL(V)$ and was found using ChatGPT 5.6 sol.
As an application, we define the \emph{Kirchhoff symmetric function} $K_G(\bx)$ of a loopless graph $G$ and prove that it is Schur-positive.

\begin{theorem}
\label{thm herm s pos proof}
Let $\bv_1,\bv_2,\dots, \bv_m, \bw_1,\bw_2,\dots, \bw_{m'}$ be a collection of vectors in $\bbC^n$. Let $A(\bz) = B(\bz) C(\bz)$, where $B(\bz) = \sum_{i = 1}^m z_i \bv_i \bv_i^*$ and $C(\bz) = \sum_{i=1}^{m'} z_{i + m} \bw_i \bw_i^*$. Then $\Omega_{F_A}(\bz, \bx)$ is Schur-positive.
\end{theorem}

\begin{proof}
Observe that the inverted roots of $F_{A}(\bz, t)$ are the eigenvalues of $A(\bz)$. Hence, the coefficient of $s_\lambda(\bx)$ in $\Omega_{F_A}(\bz, \bx)$ is $s_{\lambda'}[A(\bz)]$.

Recall that $V^{\otimes k} \cong \bigoplus_{\nu \vdash k} S_\nu(V)^{\oplus f^\nu}$. Let $P_\lambda$ be the orthogonal projection from $V^{\otimes k}$ to $S_\lambda(V)^{\oplus f^\lambda} \subseteq V^{\otimes k}$ defined in Definition~\ref{def schur proj}. 
We have that $A(\bz)^{\otimes k}$ acts on $S_\lambda(V)^{\oplus f^\lambda}$ as $S_\lambda(A(\bz))^{\oplus f^\lambda}$. Hence, $\Tr(P_\lambda A(\bz)^{\otimes k}) = \Tr(S_\lambda(A(\bz)))\cdot f^\lambda = s_\lambda[A(\bz)]\cdot f^\lambda$. 

Observe that 
\begin{align}
A(\bz)^{\otimes k} &= (B(\bz) C(\bz))^{\otimes k} = \left(\sum_{i=1}^m \sum_{j = 1}^{m'} z_i z_{m+j} v_i v_i^* w_j w_j^*\right)^{\otimes k} \\
&= \sum_{\ba \in [m]^k} \sum_{\bb \in [m']^k} \bz_\ba \bz_\bb \bv_\ba \bv_\ba^* \bw_\bb \bw_\bb^*,
\end{align}
where $\bz_\ba = z_{a_1} z_{a_2} \cdots z_{a_k}$, $\bz_\bb = z_{m + b_1} z_{m + b_2} \cdots z_{m + b_k}$, $\bv_\ba = \bv_{a_1} \otimes \bv_{a_2} \otimes \dots \otimes \bv_{a_k}$, and $\bw_\bb = \bw_{b_1} \otimes \bw_{b_2} \otimes \dots \otimes \bw_{b_k}$.

As $P_\lambda$ is an orthogonal projection, $P_\lambda = P_\lambda^* = P_\lambda^2$. Furthermore, as $P_\lambda$ is a map on $GL(V)$-modules, $P_\lambda M^{\otimes k} = M^{\otimes k} P_\lambda$ for any map $M: V \to V$. Hence,
\begin{align}
\Tr(P_\lambda A(\bz)^{\otimes k}) &= \Tr(P_\lambda A(\bz)^{\otimes k} P_\lambda) = \Tr\left(\sum_{\ba \in [m]^k} \sum_{\bb \in [m']^k} \bz_\ba \bz_\bb P_\lambda \bv_\ba \bv_\ba^* \bw_\bb \bw_\bb^* P_\lambda\right) \\
&= \sum_{\ba \in [m]^k} \sum_{\bb \in [m']^k} \bz_\ba \bz_\bb \Tr(P_\lambda \bv_\ba \bv_\ba^* \bw_\bb \bw_\bb^* P_\lambda) = \sum_{\ba \in [m]^k} \sum_{\bb \in [m']^k} \bz_\ba \bz_\bb\Tr( P_\lambda \bv_\ba \bv_\ba^* P_\lambda \bw_\bb \bw_\bb^*).
\end{align}
As $\Tr( P_\lambda \bv_\ba \bv_\ba^* P_\lambda \bw_\bb \bw_\bb^*) = \Tr(\bw_\bb^* P_\lambda \bv_\ba \bv_\ba^* P_\lambda \bw_\bb) = (\bw_\bb^* P_\lambda \bv_\ba) \cdot (\bv_\ba^* P_\lambda \bw_\bb) = (\bw_\bb^* P_\lambda \bv_\ba)^2$, we have $\Tr(P_\lambda A(\bz)^{\otimes k}) \in \bbR_{\geq 0}[\bz]$. Therefore $s_\lambda[A(\bz)] \in \bbR_{\geq 0}[\bz]$, and $\Omega_{F_A}(\bz, \bx)$ is Schur-positive.
\end{proof}

\begin{corollary}[Theorem~\ref{thm herm s pos}]
Let $A(\bz) = \sum_{i = 1}^m z_i \bv_i \bv_i^*$ for vectors $\bv_1,\bv_2,\dots, \bv_m \in \bbC^n$. Then $\Omega_{F_A}(\bz, \bx)$ is Schur-positive. Moreover, for a partition $\lambda \vdash k$, the coefficient of $s_{\lambda'}(\bx)$ is 
\begin{equation}
s_\lambda[A(\bz)] = \frac{1}{f^\lambda} \sum_{\ba \in [m]^k} \bz_\ba \|P_\lambda \bv_\ba\|^2 = \frac{1}{f^\lambda} \sum_{\ba \in [m]^k} \bz_\ba \imm_\lambda(K_\ba),
\end{equation}
where $K_\ba$ is the Gram matrix $\left[\bv_{a_i}^*\bv_{a_j}\right]_{i,j = 1}^k$.
\end{corollary}
\begin{proof}
From the proof of Theorem~\ref{thm herm s pos proof}, we have 
\begin{equation}
\Tr(P_\lambda A(\bz)^{\otimes k}) = f^\lambda \cdot s_\lambda[A(\bz)] = \sum_{\ba \in [m]^k} \bz_\ba \Tr(P_\lambda \bv_\ba \bv_\ba^*).
\end{equation}
Observe that as $P_\lambda = P_\lambda^* = P_\lambda^2$,
\begin{equation}
\Tr(P_\lambda \bv_\ba \bv_\ba^*) = \bv_\ba^* P_\lambda \bv_\ba = \bv_\ba^* P_\lambda^* P_\lambda \bv_\ba = (P_\lambda \bv_\ba)^*(P_\lambda \bv_\ba) = \|P_\lambda \bv_\ba\|^2.
\end{equation}
From Theorem~\ref{thm norm imm}, $\|P_\lambda \bv_\ba\|^2 = \imm_\lambda(K_\ba)$ and the result follows.
\end{proof}

A combinatorially interesting application of Theorem~\ref{thm herm s pos proof} is the case when $A(\bz)$ is the \emph{weighted graph Laplacian} $\mathcal{L}_G(\bz)$ of a graph $G$. The graph Laplacian is a central object of study in spectral graph theory.

\begin{defin}
Let $G$ be a loopless, not necessarily simple graph on vertex set $[n]$. For each edge $e = \{i < j\} \in E(G)$, define $\bv_e = \mathbf{1}_i - \mathbf{1}_j \in \mathbb{R}^n$, where $\mathbf{1}_i$ is the $i^{th}$ standard basis vector. Define the \emph{weighted graph Laplacian} $\mathcal{L}_G(\bz)$ by
\begin{equation}
\mathcal{L}_G(\bz) = \sum_{e \in E(G)} z_e \bv_e \bv_e^*.
\end{equation}
\end{defin}

Kirchhoff's \emph{matrix-tree theorem} gives a combinatorial interpretation of principal minors of $\mathcal{L}_G(\bz)$. We state a general form of Kirchhoff's theorem below.

\begin{defin}
A \emph{rooted spanning forest} of a graph $G$ is a pair of sets $A \subseteq E(G)$ and $R \subseteq V(G)$ such that $A$ contains no cycles and $R$ has one vertex from each connected component of $A$. 
Write $\mathcal{RF}(G)$ for the set of rooted spanning forests of $G$. Define the \emph{rooted forest polynomial} of $G$ by
\begin{equation}
RF_G(\bz, t) = \sum_{(A, R) \in \mathcal{RF}(G)} t^{|A|} \prod_{e \in A} z_e.
\end{equation}
\end{defin}

\begin{theorem}
\label{thm matrix-tree}
\cite[Theorem 4.7]{Bapat}
Let $G$ be a loopless graph. Then 
\begin{equation}
RF_G(\bz, t) = \det(I + t \mathcal{L}_G(\bz)).
\end{equation}
\end{theorem}

Using Theorem~\ref{thm matrix-tree}, we obtain a combinatorial interpretation for the monomial expansion of $\Omega_{RF_G}(\bz, \bx)$. We state this in terms of the \emph{Kirchhoff symmetric function}, which we introduce below. As a corollary of Theorem~\ref{thm herm s pos proof}, we obtain Schur-positivity of Kirchhoff symmetric functions. A natural problem is to determine a combinatorial interpretation of the Schur-coefficients.

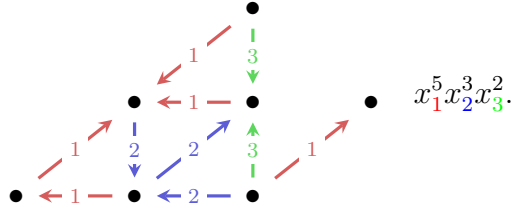
\begin{figure}
\caption{An element of $\Pi_{\mathcal{RF}(G)}(5,3,2)$ with its associated weight.}
`\usetikzlibrary{arrows,arrows.meta}
\tikzcdset{arrow style=tikz, diagrams={>= stealth}}
\[\begin{tikzcd}[ampersand replacement=\&, every arrow/.append style={line width=1.2pt}]
	\&\& \bullet \& \\
	\& \bullet \& \bullet \& \bullet \\
	\bullet \& \bullet \& \bullet
	\arrow["1"{description}, color={rgb,255:red,214;green,92;blue,92}, from=1-3, to=2-2]
	\arrow["3"{description}, color={rgb,255:red,92;green,214;blue,92}, from=1-3, to=2-3]
	\arrow["2"{description}, color={rgb,255:red,92;green,92;blue,214}, from=2-2, to=3-2]
	\arrow["1"{description}, color={rgb,255:red,214;green,92;blue,92}, from=2-3, to=2-2]
	\arrow["1"{description}, color={rgb,255:red,214;green,92;blue,92}, from=3-1, to=2-2]
	\arrow["2"{description}, color={rgb,255:red,92;green,92;blue,214}, from=3-2, to=2-3]
	\arrow["1"{description}, color={rgb,255:red,214;green,92;blue,92}, from=3-2, to=3-1]
	\arrow["3"{description}, color={rgb,255:red,92;green,214;blue,92}, from=3-3, to=2-3]
	\arrow["1"{description}, color={rgb,255:red,214;green,92;blue,92}, from=3-3, to=2-4]
	\arrow["2"{description}, color={rgb,255:red,92;green,92;blue,214}, from=3-3, to=3-2]
\end{tikzcd}
\ \ x^5_{\color{red}{1}} x^3_{\color{blue}{2}} x^2_{\color{green}{3}}.
\]

\end{figure}

\begin{defin}
Let $G$ be a graph with $m$ edges.
For a partition $\lambda \vdash m$, let $\Pi_{\mathcal{RF}(G)}(\lambda)$ be the sequence $((A_i, R_i))_{i \geq 1}$ of rooted spanning forests in $G$ such that $|A_i| = \lambda_i$ and $\bigcup A_i = E(G).$ 
Define the \emph{Kirchhoff symmetric function} $K_G(\bx)$ by
\begin{equation}
K_G(\bx) = \sum_{\lambda \vdash m} |\Pi_{\mathcal{RF}(G)}(\lambda)|\, m_\lambda(\bx).
\end{equation}
\end{defin}

\begin{theorem}
For $\alpha \in \mathbb{Z}_{\geq 0}^{E}$, let $G^{\alpha}$ be the graph obtained from $G$ by taking $\alpha_e$ copies of each edge $e$. Then
\begin{equation}
\Omega_{RF_G}(\bz, \bx) = \sum_{\alpha \in \bbZ^E_{\geq 0}} \frac{\bz^\alpha K_{G^\alpha}(\bx)}{\alpha!},
\end{equation} 
where $\alpha! = \prod_{e \in E} \alpha_e !.$
\end{theorem}
\begin{proof}
The coefficient of $m_\lambda(\bx)$ in $\Omega_{RF_G}(\bz, \bx)$ is $e_\lambda^{RF_G}(\bz)$. We have
\begin{equation}
e_{\lambda}^{RF_G}(\bz) = \prod_{i = 1}^{\ell(\lambda)} \sum_{\substack{(A, R) \in \mathcal{RF}(G) \\ |A| = \lambda_i}} \prod_{e \in A} z_e. 
\end{equation} 
Hence, for $\alpha \in \bbZ_{\geq 0}^E$, the coefficient of $\bz^\alpha$ in $e_\lambda^{RF_G}$ is the number of sequences $((A_i, R_i))_{i=1}^{\ell(\lambda)}$ of rooted forests such that $|A_i| = \lambda_i$ and $\prod_{i = 1}^{\ell(\lambda)} \prod_{e \in A_i} z_e = \bz^\alpha$. As $G^\alpha$ is obtained from $G$ by taking $\alpha_e$ copies of each edge $e$, there is a $\prod_{e \in E(G)} \alpha_e!$-to-one map from $\Pi_{\mathcal{RF}(G)}(\lambda)$ to such sequences. Thus, the coefficient of $m_\lambda(\bx)$ in $K_{G^\alpha}(\bx)$ is $\alpha!$ times the coefficient of $\bz^{\alpha}$ in $e_\lambda^{RF_G}(\bz)$, and the result follows.
\end{proof}

\begin{corollary}
\label{cor KG s pos}
For any loopless graph $G$, the Kirchhoff symmetric function $K_G(\bx)$ is Schur-positive.
\end{corollary}

%
%

\section{Concluding Remarks and Open Problems}
In this section, we state open problems arising from this paper. In addition, we consider the case when $\mathcal{R}$ is the $S_n$ group ring $\bbQ S_n$ and conjecture Schur-positivity of symmetric functions coming from Eulerian elements.

\subsection{Symmetric Function Evaluation Inequalities}
In Section~\ref{sec sym evals}, we considered the relationship between positivity of $\Omega_F(\bz, \bx)$ and symmetric function evaluations. Let $F(\bz, t) \in \bbQ[z_1,z_2,\dots, z_n, t]$. Theorem~\ref{thm sym eval ineq proof} shows that if $\Omega_F(\bz, \bx)$ is Schur-positive, then for $\br \geq_{entry} \bs \in \bbR^n$, the inverted roots of $F(\br, t)$ log-majorize the inverted roots of $F(\bs, t)$. 
As a consequence of the Schur-positivity of $X_{\inc(P)}(\bx)$ for \threeone -free posets, Corollary~\ref{cor 3+1 inv root log maj} states that the inverted roots of the independence polynomial $I_{\inc(P)}(\br, t)$ log-majorize the inverted roots of $I_{\inc(P)}(\bs, t)$ whenever $\br \geq_{entry} \bs \in \bbR_{\geq 0}^n$. 
In light of the recent counterexamples to the Schur-positivity of $X_G(\bx)$ for claw-free graphs \cite{MMclaw, PrajapatiClaw}, it is interesting to consider if the inverted root log-majorization inequalities extend to independence polynomials of claw-free graphs.
\begin{problem}
Let $G$ be a claw-free graph on $[n]$.
Determine if the inverted roots of $I_G(\br, t)$ log-majorize the inverted roots of $I_G(\bs, t)$ for all $\br \geq_{entry} \bs \in \bbR_{\geq 0}^n$.
\end{problem} 

For $S \subseteq [n]$, let $\bu_S$ be the inverted roots of $F(\mathbf{1}_S, t)$.
In Theorem~\ref{thm IncExc ineq}, we showed that if $\Omega_F(\bz, \bx)$ is Schur-positive, then for all partitions $\lambda$, the alternating sum $\sum_{S \subseteq [n]} (-1)^{n-|S|} s_\lambda(\bu_S)$ is non-negative. Likewise, if $\Omega_F(\bz, \bx)$ is $e$-positive, then $\sum_{S \subseteq [n]} (-1)^{n-|S|} m_\lambda(\bu_S)$ is non-negative for all partitions $\lambda$. This suggests the following problems.
In each case, the $n = 1$ case is solved by Theorem~\ref{thm sym eval ineq proof}, as
$s_\lambda(\bu_{[1]}) - s_\lambda(\bu_\emptyset) \geq 0$ for all $\lambda$ and
$m_\lambda(\bu_{[1]}) - m_\lambda(\bu_\emptyset) \geq 0$ for all $\lambda$ are both equivalent to $\bu_{[1]} \succeq_{\log} \bu_\emptyset$.

\begin{problem}
Characterize functions $\bu: 2^{[n]} \to \mathbb{R}_{\geq 0}^k$ such that 
\begin{equation}
\sum_{S \subseteq [n]} (-1)^{n-|S|} s_\lambda(\bu_S) \geq 0.
\end{equation}
\end{problem}

\begin{problem}
Characterize functions $\bu: 2^{[n]} \to \mathbb{R}_{\geq 0}^k$ such that 
\begin{equation}
\sum_{S \subseteq [n]} (-1)^{n-|S|} m_\lambda(\bu_S) \geq 0.
\end{equation}
\end{problem}

\subsection{Schur-Positivity from Networks}
In Section~\ref{sec s pos from networks}, we used Lindstr\"om's lemma to give a new proof of Schur-positivity of $X_{\inc(P)}(\bx)$ when $P$ is a natural unit interval order. In particular, for a Hessenberg function $\bh \in \mathbb{H}_n$, we constructed a directed network $\Gamma_\bh$ with weights in $\bbQ_{\geq 0}[z_1,z_2,\dots, z_n]$ and weight matrix $[e_{j-i}^{\inc(P_\bh)}(\bz)]_{i,j=1}^\infty$. Interestingly, there does not seem to exist a directed network $\Gamma$ with weights in $\bbQ_{\geq 0}[z_1,z_2,\dots,z_n]$ and weight matrix $[e_{j-i}^{\inc(P)}(\bz)]_{i,j = 1}^\infty$ for general \threeone -free posets $P$. In particular, the author was unable to find such a directed network for $P = \twotwo$. This suggests the following problem.

\begin{problem}
Let $\Gamma$ be directed network such that the weight of each source-to-sink path is a polynomial in $\bbQ_{\geq 0}[z_1,z_2,\dots, z_n]$. Under what conditions will the weight matrix $W_\Gamma$ be totally nonnegative with respect to $\bbQ_{\geq 0}[z_1,z_2,\dots, z_n]$?  
\end{problem}

\subsection{Kirchhoff Symmetric Functions}

In Section~\ref{sec SPD matrices}, we introduce the Kirchhoff symmetric function $K_G(\bx)$ of a loopless graph $G$. As a consequence of Theorem~\ref{thm herm s pos proof}, Corollary~\ref{cor KG s pos} shows that $K_G(\bx)$ is Schur-positive for any loopless graph $G$. A natural problem is to find a combinatorial interpretation for the Schur-coefficients of $K_G(\bx)$.

\begin{problem}
Find a combinatorial interpretation of the Schur-coefficients of the Kirchhoff symmetric function $K_G(\bx)$.
\end{problem}

\subsection{Symmetric Group Statistics}
In this section, we consider the case when $\mathcal{R}$ is the symmetric group algebra $\mathbb{Q}S_n$ and $\mathcal{R}_+$ is $\mathbb{Q}_{\geq 0}S_n$. In particular, we consider polynomials of the form 
\begin{equation}
F(t) = \sum_{\pi \in S_n} \pi t^{\stat(\pi)}
\end{equation}
for a symmetric group statistic $\stat: S_n \to \bbZ_{\geq 0}$.
For $\Omega_F(\bx)$ to be a well-defined power series, we require that $\stat(\pi) = 0$ if and only if $\pi$ is the identity in $S_n$. Likewise, for $\Omega_F(\bx)$ to be a symmetric function, we require that the coefficients of $F(t)$ commute with each other in $\mathbb{Q}S_n$. As a corollary of Theorem~\ref{thm stan poly s pos}, if $\Omega_F(\bx)$ is Schur-positive, then the polynomial $\sum_{\pi \in S_n} t^{\stat(\pi)}$, obtained from $F(t)$ by sending each permutation in $S_n$ to $1$, must be real-rooted. Statistics satisfying these conditions are somewhat rare. We consider the cases when $\stat(\pi)$ is $n - \cyc(\pi)$, $\des(\pi)$, and $\lpk(\pi)$, where $\cyc(\pi)$ is the number of cycles, $\des(\pi)$ is the number of descents, and $\lpk(\pi)$ is the number of left peaks. 


\begin{theorem}\cite{Jucys}
The Jucys--Murphy elements commute in $\mathbb{Q}S_n$. Moreover, 
\begin{equation}
C_n(t) := \sum_{\pi \in S_n} \pi t^{n -\cyc(\pi)} = \prod_{k = 2}^n (J_k t + 1),
\end{equation}
where $J_k = \sum_{1 \leq i < k} \tau_{i,k}$ is the $k^{th}$ Jucys--Murphy element.
\end{theorem}

\begin{corollary}
The symmetric function $\Omega_{C_n}(\bx)$ is $e$-positive.
\end{corollary}
\begin{proof}
From the classical dual Cauchy product formula, $\Omega_{C_n}(\bx) = \sum_{\lambda} m_\lambda(J_2,J_3,\dots, J_n) e_\lambda(\bx)$. As each $J_i \in \mathcal{R}_+$, we have $m_\lambda(J_2, J_3,\dots, J_n) \in \mathcal{R}_+$ for all partitions $\lambda$, so $\Omega_{C_n}(\bx)$ is $e$-positive. 
\end{proof}

One of the most well-known family of real-rooted polynomials in combinatorics are the \emph{Eulerian polynomials}. The \emph{$n^{th}$-Eulerian polynomial} $A_n(t)$ is defined by
\begin{equation}
A_n(t) = \sum_{\pi \in S_n} t^{\des(\pi)},
\end{equation}
where $\des(\pi) = \#\{i \in [n-1] \mid \pi_i > \pi_{i+1}\}$ is the \emph{descent statistic}. We define the \emph{$\bbQ S_n$-coefficient Eulerian polynomial} $E_n(t) \in \bbQ S_n[t]$ by
\begin{equation}
E_n(t) = \sum_{i = 0}^{n-1} E_n^{(i)} t^i := \sum_{\pi \in S_n} \pi t^{\des(\pi)}.
\end{equation}
A theorem of Loday \cite{Loday} states that the \emph{Eulerian elements} $\{E_n^{(i)}\}_{i = 1}^{n-1}$ commute with each other.
\begin{theorem}
\label{thm Loday}
\cite{Loday}
For each $0 \leq i,j \leq n-1$,
\begin{equation}
E_n^{(i)} E_n^{(j)} = E_n^{(j)} E_n^{(i)}.
\end{equation}
\end{theorem}
As the identity element is the only permutation with 0 descents and the coefficients of $E_n(t)$ commute, $\Omega_{E_n}(\bx)$ is a well-defined symmetric function with coefficients in $\mathbb{Q}S_n$. 
\begin{conjecture}
\label{conj Eul elem s pos}
The symmetric function $\Omega_{E_n}(\bx) = \prod_{i \geq 1} E_n(x_i)$ is Schur-positve.
\end{conjecture}

For a permutation $\pi \in S_n$, a \emph{left peak} of $\pi$ is an index $1 \leq i < n$ such that $\pi_{i-1} < \pi_i > \pi_{i+1}$ with the convention that $\pi_0 = 0$. Define the \emph{left peak statistic} $\lpk(\pi)$ to be the number of left peaks of $\pi$. Write $L_n^{(i)} = \sum_{\pi \in S_n: \lpk(\pi) = i} \pi$ for the \emph{$i^{th}$ left peak Eulerian element} of $\mathbb{Q}S_n$. Petersen \cite{Petersen} showed that the  left peak Eulerian elements commute with each other. 
\begin{theorem}
\label{thm Petersen}
\cite{Petersen}
For each $0 \leq i, j \leq  n-1$,
\begin{equation}
L_n^{(i)} L_n^{(j)} = L_n^{(j)} L_n^{(i)}.
\end{equation}
\end{theorem}  
Petersen also related the polynomial $\sum_{\pi \in S_n} t^{\lpk(\pi)}$ to the type B Eulerian polynomial \cite[Proposition 4.15]{Petersen}. Consequently, 
$\sum_{\pi \in S_n} t^{\lpk(\pi)}$ is real-rooted. As the identity permutation is the only permutation with 0 left peaks, $\Omega_{L_n}(\bx)$ is a well-defined symmetric function for $L_n(t) = \sum_{i = 0}^{n-1} L_n^{(i)} t^i$.
We conjecture that $\Omega_{L_n}(\bx)$ is Schur-positive.
\begin{conjecture}
\label{conj peak eul s pos}
The symmetric function $\Omega_{L_n}(\bx) = \prod_{i \geq 1} L_n(x_i)$ is Schur-positive.
\end{conjecture}

A major obstacle to understanding the symmetric functions $\Omega_{E_n}(\bx)$ and $\Omega_{L_n}(\bx)$ is the lack of bijective proofs of Theorems~\ref{thm Loday} and \ref{thm Petersen}. Write $\mathbf{E}_n^{(i)} = \{\pi \in S_n \mid \des(\pi) = i\}$ and $\mathbf{L}_n^{(i)} = \{\pi \in S_n \mid \lpk(\pi) = i\}$. 
Solutions to the following open problems would be helpful for understanding $\Omega_{E_n}(\bx)$ and $\Omega_{L_n}(\bx)$.
\begin{problem}
Find bijections $\varphi_E^{i,j}: \mathbf{E}_n^{(i)} \times \mathbf{E}_n^{(j)} \to \mathbf{E}_n^{(j)} \times \mathbf{E}_n^{(i)}$ such that if $\varphi_E^{i,j}(\pi, \sigma) = (\sigma', \pi')$, then $\pi\sigma = \sigma'\pi'$. 
\end{problem}
\begin{problem}
Find bijections $\varphi_L^{i,j}: \mathbf{L}_n^{(i)} \times \mathbf{L}_n^{(j)} \to \mathbf{L}_n^{(j)} \times \mathbf{L}_n^{(i)}$ such that if $\varphi_L^{i,j}(\pi, \sigma) = (\sigma', \pi')$, then $\pi\sigma = \sigma'\pi'$. 
\end{problem}

\section{Acknowledgements}
The author thanks his advisor Sara Billey for her mentorship.
The author thanks Nantel Bergeron, Jonah Blasiak, Tracy Chin, Darij Grinberg, Joe Rogge, Hunter Spink, Vasu Tewari, and Cynthia Vinzant for helpful conversations.

\section{AI Use Statement}
The proof of Theorem~\ref{thm herm s pos} was found using ChatGPT 5.6 sol.

\clearpage
\bibliographystyle{plain}
\bibliography{mycitations}

@book{Bapat,
 author = {Bapat, R. B.},
 title = {Graphs and matrices},
 fseries = {Universitext},
 series = {Universitext},
 issn = {0172-5939},
 isbn = {978-1-84882-980-0; 978-1-84882-981-7},
 year = {2011},
 publisher = {London: Springer; New Delhi: Hindustan Book Agency},
 language = {English},
 doi = {10.1007/978-1-84882-981-7},
 zbMATH = {5712789},
 Zbl = {1248.05002}
}

@Article{BLamLLT,
    AUTHOR = {Blasiak, Jonah},
     TITLE = {Haglund's conjecture on 3-column {M}acdonald polynomials},
   JOURNAL = {Math. Z.},
  FJOURNAL = {Mathematische Zeitschrift},
    VOLUME = {283},
      YEAR = {2016},
    NUMBER = {1-2},
     PAGES = {601--628},
      ISSN = {0025-5874},
   MRCLASS = {05E05 (33C52)},
  MRNUMBER = {3489082},
MRREVIEWER = {Meesue Yoo},
       DOI = {10.1007/s00209-015-1612-7},
       URL = {http://dx.doi.org/10.1007/s00209-015-1612-7},
}

@misc{OAIePos,
  author       = {{OpenAI}},
  title        = {{Elementary positivity of chromatic quasisymmetric functions}},
  year         = {2026},
  url   = {https://github.com/openai/math/tree/main/preprints/Elementary-Positivity-of-Chromatic-Quasisymmetric-Functions-September-24-2026}
}

@ARTICLE{BEPS,
   author = {{Blasiak}, Jonah and {Eriksson}, Holden and {Pylyavskyy}, Pavlo and {Siegl}, Isaiah},
   title = "{Noncommutative Schur functions for posets}",
  FJournal = {Selecta Mathematica. New Series},
  Journal = {Selecta Math. (N.S.)},
  ISSN = {1420-9020},
  Volume = {31},
  Number = {21},
  Pages = {56},
  year = {2025}
}

@incollection{Tymoczko,
 author = {Tymoczko, Julianna S.},
 title = {Permutation actions on equivariant cohomology of flag varieties},
 booktitle = {Toric topology. International conference, Osaka, Japan, May 28--June 3, 2006},
 isbn = {978-0-8218-4486-1},
 pages = {365--384},
 year = {2008},
 publisher = {Providence, RI: American Mathematical Society (AMS)},
 language = {English},
 zbMATH = {5343262},
 Zbl = {1147.14024}
}

@article{BLKronecker,
    AUTHOR = {Blasiak, Jonah and Liu, Ricky Ini},
     TITLE = {Kronecker coefficients and noncommutative super {S}chur
              functions},
   JOURNAL = {J. Combin. Theory Ser. A},
  FJOURNAL = {Journal of Combinatorial Theory. Series A},
    VOLUME = {158},
      YEAR = {2018},
     PAGES = {315--361},
      ISSN = {0097-3165},
   MRCLASS = {05E10 (20C30)},
  MRNUMBER = {3800131},
       DOI = {10.1016/j.jcta.2018.02.007},
       URL = {https://doi.org/10.1016/j.jcta.2018.02.007},
}

@article {LamRibbon,
    AUTHOR = {Lam, Thomas},
     TITLE = {Ribbon {S}chur operators},
   JOURNAL = {European J. Combin.},
  FJOURNAL = {European Journal of Combinatorics},
    VOLUME = {29},
      YEAR = {2008},
    NUMBER = {1},
     PAGES = {343--359},
      ISSN = {0195-6698},
   MRCLASS = {05E10 (17B37)},
  MRNUMBER = {2368641 (2009a:05210)},
MRREVIEWER = {Arthur L. B. Yang},
       DOI = {10.1016/j.ejc.2006.01.016},
       URL = {http://dx.doi.org/10.1016/j.ejc.2006.01.016},
}

@Article{Hwang,
 Author = {Hwang, Byung-Hak},
 Title = {Chromatic quasisymmetric functions and noncommutative {{\(P\)}}-symmetric functions},
 FJournal = {Transactions of the American Mathematical Society},
 Journal = {Trans. Am. Math. Soc.},
 ISSN = {0002-9947},
 Volume = {377},
 Number = {4},
 Pages = {2855--2896},
 Year = {2024},
 Language = {English},
 DOI = {10.1090/tran/9096},
 zbMATH = {7876035},
 Zbl = {1542.05177}
}

@article{Jucys,
title = {Symmetric polynomials and the center of the symmetric group ring},
journal = {Reports on Mathematical Physics},
volume = {5},
number = {1},
pages = {107-112},
year = {1974},
issn = {0034-4877},
doi = {https://doi.org/10.1016/0034-4877(74)90019-6},
url = {https://www.sciencedirect.com/science/article/pii/0034487774900196},
author = {A.-A.A. Jucys}
}

@article {ASW,
    AUTHOR = {Aissen, Michael and Schoenberg, I. J. and Whitney, A. M.},
     TITLE = {On the generating functions of totally positive sequences.
              {I}},
   JOURNAL = {J. Analyse Math.},
  FJOURNAL = {Journal d'Analyse Math\'ematique},
    VOLUME = {2},
      YEAR = {1952},
     PAGES = {93--103},
      ISSN = {0021-7670,1565-8538},
   MRCLASS = {27.0X},
  MRNUMBER = {53174},
MRREVIEWER = {E.\ Hille},
       DOI = {10.1007/BF02786970},
       URL = {https://doi.org/10.1007/BF02786970},
}

@article{Edrei,
 author = {Edrei, Albert},
 title = {On the generating functions of totally positive sequences. {II}},
 fjournal = {Journal d'Analyse Math{\'e}matique},
 journal = {J. Anal. Math.},
 issn = {0021-7670},
 volume = {2},
 pages = {104--109},
 year = {1952},
 language = {English},
 doi = {10.1007/BF02786971},
 zbMATH = {3077239},
 Zbl = {0049.17202}
}

@article{Thoma,
 author = {Thoma, Elmar},
 title = {Die unzerlegbaren, positiv-definiten {Klassenfunktionen} der abz{\"a}hlbar unendlichen, symmetrischen {Gruppe}},
 fjournal = {Mathematische Zeitschrift},
 journal = {Math. Z.},
 issn = {0025-5874},
 volume = {85},
 pages = {40--61},
 year = {1964},
 language = {German},
 doi = {10.1007/BF01114877},
 url = {https://eudml.org/doc/170284},
 zbMATH = {3305204},
 Zbl = {0192.12402}
}

@article{Lindstrom,
 author = {Lindstr{\"o}m, Bernt},
 title = {On the vector representations of induced matroids},
 fjournal = {Bulletin of the London Mathematical Society},
 journal = {Bull. Lond. Math. Soc.},
 issn = {0024-6093},
 volume = {5},
 pages = {85--90},
 year = {1973},
 language = {English},
 doi = {10.1112/blms/5.1.85},
 zbMATH = {3412660},
 Zbl = {0262.05018}
}

@misc{AmdShaSta,
      title={Sprout Symmetric Functions: Part 1}, 
      author={Tewodros Amdeberhan and John Shareshian and Richard P. Stanley},
      year={2026},
      howpublished = {Preprint, {arXiv}:2605.27828},
      eprint={2605.27828},
      archivePrefix={arXiv},
      primaryClass={math.CO},
      url={https://arxiv.org/abs/2605.27828}, 
}

@misc{AlexanderssonDai,
 author = {Per Alexandersson and Lilan Dai},
 title = {Partition division maps, symmetric functions and positivity},
 year = {2026},
 howpublished = {Preprint, {arXiv}:2604.25440},
 url = {https://arxiv.org/abs/2604.25440},
 arXiv = {arXiv:2604.25440}
}

@article{BrentiPF,
 author = {Brenti, Francesco},
 title = {Combinatorics and total positivity},
 fjournal = {Journal of Combinatorial Theory. Series A},
 journal = {J. Comb. Theory, Ser. A},
 issn = {0097-3165},
 volume = {71},
 number = {2},
 pages = {175--218},
 year = {1995},
 language = {English},
 doi = {10.1016/0097-3165(95)90000-4},
 zbMATH = {802355},
 Zbl = {0851.05095}
}

@article{DeanVerducci1990,
 author = {Dean, A. M. and Verducci, J. S.},
 title = {Linear transformations that preserve majorization, {Schur} concavity, and exchangeability},
 fjournal = {Linear Algebra and its Applications},
 journal = {Linear Algebra Appl.},
 issn = {0024-3795},
 volume = {127},
 pages = {121--138},
 year = {1990},
 language = {English},
 doi = {10.1016/0024-3795(90)90339-E},
 zbMATH = {4171108},
 Zbl = {0712.15002}
}

@article{Ando1989,
 author = {Ando, T.},
 title = {Majorization, doubly stochastic matrices, and comparison of eigenvalues},
 fjournal = {Linear Algebra and its Applications},
 journal = {Linear Algebra Appl.},
 issn = {0024-3795},
 volume = {118},
 pages = {163--248},
 year = {1989},
 language = {English},
 doi = {10.1016/0024-3795(89)90580-6},
 zbMATH = {4101381},
 Zbl = {0673.15011}
}

@article{ChudnovSeym,
 author = {Chudnovsky, Maria and Seymour, Paul},
 title = {The roots of the independence polynomial of a clawfree graph},
 fjournal = {Journal of Combinatorial Theory. Series B},
 journal = {J. Comb. Theory, Ser. B},
 issn = {0095-8956},
 volume = {97},
 number = {3},
 pages = {350--357},
 year = {2007},
 language = {English},
 doi = {10.1016/j.jctb.2006.06.001},
 zbMATH = {5143573},
 Zbl = {1119.05075}
}

@article{Engstrom,
 author = {Engstr{\"o}m, Alexander},
 title = {Inequalities on well-distributed point sets on circles},
 fjournal = {JIPAM. Journal of Inequalities in Pure \& Applied Mathematics},
 journal = {JIPAM, J. Inequal. Pure Appl. Math.},
 issn = {1443-5756},
 volume = {8},
 number = {2},
 pages = {5},
 note = {Id/No 34},
 year = {2007},
 language = {English},
 zbMATH = {5224744},
 Zbl = {1370.52007}
}

@article{LeakeRyder,
 author = {Leake, Jonathan D. and Ryder, Nick R.},
 title = {Generalizations of the matching polynomial to the multivariate independence polynomial},
 fjournal = {Algebraic Combinatorics},
 journal = {Algebr. Comb.},
 issn = {2589-5486},
 volume = {2},
 number = {5},
 pages = {781--802},
 year = {2019},
 language = {English},
 doi = {10.5802/alco.63},
 zbMATH = {7115041},
 Zbl = {1434.05075}
}

@misc{MMclaw,
      title={Chromatic symmetric functions of claw-free graphs are not Schur positive}, 
      author={Jacob P. Matherne and Alejandro H. Morales},
      year={2026},
      eprint={2607.21508},
      archivePrefix={arXiv},
      howpublished = {Preprint, {arXiv}:2607.21508},
      primaryClass={math.CO},
      url={https://arxiv.org/abs/2607.21508}, 
}

@misc{PrajapatiClaw,
      title={A counterexample to the claw-free Schur-positivity conjecture}, 
      author={Jitendra Prajapati},
      year={2026},
      eprint={2607.26364},
      howpublished = {Preprint, {arXiv}:2607.26364},
      archivePrefix={arXiv},
      primaryClass={math.CO},
      url={https://arxiv.org/abs/2607.26364}, 
}

@misc{WZZclaw,
      title={Two infinite families of counterexamples to the Stanley--Gasharov conjecture}, 
      author={David G. L. Wang and K. Zhang and T. Y. Zhao},
      year={2026},
      eprint={2607.27166},
      archivePrefix={arXiv},
      howpublished = {Preprint, {arXiv}:2607.27166},
      primaryClass={math.CO},
      url={https://arxiv.org/abs/2607.27166}, 
}

@book{FultonHarris,
 author = {Fulton, William and Harris, Joe},
 title = {Representation theory. {A} first course},
 fseries = {Graduate Texts in Mathematics},
 series = {Grad. Texts Math.},
 issn = {0072-5285},
 volume = {129},
 isbn = {0-387-97495-4; 0-387-97527-6},
 year = {1991},
 publisher = {New York etc.: Springer-Verlag},
 language = {English},
 zbMATH = {51906},
 Zbl = {0744.22001}
}

@book{MerrisMultilin,
 author = {Merris, Russell},
 title = {Multilinear algebra},
 fseries = {Algebra, Logic and Applications},
 series = {Algebra Log. Appl.},
 issn = {1041-5394},
 volume = {8},
 isbn = {90-5699-078-0},
 year = {1997},
 publisher = {Langhorne, PA: Gordon \& Breach},
 language = {English},
 zbMATH = {1050306},
 Zbl = {0892.15020}
}

@article{Petersen,
 author = {Petersen, T. Kyle},
 title = {Enriched {{\(P\)}}-partitions and peak algebras},
 fjournal = {Advances in Mathematics},
 journal = {Adv. Math.},
 issn = {0001-8708},
 volume = {209},
 number = {2},
 pages = {561--610},
 year = {2007},
 language = {English},
 doi = {10.1016/j.aim.2006.05.016},
 zbMATH = {5129762},
 Zbl = {1111.05097}
}

@article{Loday,
 author = {Loday, Jean-Louis},
 title = {Op{\'e}rations sur l'homologie cyclique des alg{\`e}bres commutatives. ({Operations} on the cyclic homology of commutative algebras)},
 fjournal = {Inventiones Mathematicae},
 journal = {Invent. Math.},
 issn = {0020-9910},
 volume = {96},
 number = {1},
 pages = {205--230},
 year = {1989},
 language = {French},
 doi = {10.1007/BF01393976},
 url = {https://eudml.org/doc/143678},
 zbMATH = {4123962},
 Zbl = {0686.18006}
}

@article{Muirhead,
 author = {Muirhead, R. F.},
 title = {Some methods applicable to identities and inequalities of symmetric algebraic functions of {{\(n\)}} letters.},
 fjournal = {Proceedings of the Edinburgh Mathematical Society},
 journal = {Proc. Edinb. Math. Soc.},
 issn = {0013-0915},
 volume = {21},
 pages = {143--157},
 year = {1903},
 language = {English},
 zbMATH = {2656169},
 JFM = {34.0202.01}
}

@article{BF,
    AUTHOR = {Blasiak, Jonah and Fomin, Sergey},
     TITLE = {Noncommutative {S}chur functions, switchboards, and {S}chur
              positivity},
   JOURNAL = {Selecta Math. (N.S.)},
  FJOURNAL = {Selecta Mathematica. New Series},
    VOLUME = {23},
      YEAR = {2017},
    NUMBER = {1},
     PAGES = {727--766},
      ISSN = {1022-1824},
   MRCLASS = {05E05 (05A15 16S99)},
  MRNUMBER = {3595905},
       DOI = {10.1007/s00029-016-0253-y},
       URL = {http://dx.doi.org/10.1007/s00029-016-0253-y},
}

@article {BFcommute,
    AUTHOR = {Blasiak, Jonah and Fomin, Sergey},
     TITLE = {Rules of three for commutation relations},
   JOURNAL = {J. Algebra},
  FJOURNAL = {Journal of Algebra},
    VOLUME = {500},
      YEAR = {2018},
     PAGES = {193--220},
      ISSN = {0021-8693},
   MRCLASS = {05E05 (16S99)},
  MRNUMBER = {3765453},
MRREVIEWER = {Matt S. Davis},
       DOI = {10.1016/j.jalgebra.2016.12.023},
       URL = {https://doi.org/10.1016/j.jalgebra.2016.12.023},
}

@article{KhareTao,
 author = {Khare, Apoorva and Tao, Terence},
 title = {Schur polynomials, entrywise positivity preservers, and weak majorization},
 fjournal = {S{\'e}minaire Lotharingien de Combinatoire},
 journal = {S{\'e}min. Lothar. Comb.},
 issn = {1286-4889},
 volume = {80B},
 pages = {12},
 note = {Id/No 14},
 year = {2018},
 language = {English},
 url = {www.mat.univie.ac.at/~slc/wpapers/FPSAC2018//14-Khare-Tao.html},
 zbMATH = {7054603},
 Zbl = {1411.05263}
}

@book {ec2,
    AUTHOR = {Stanley, Richard P.},
     TITLE = {Enumerative combinatorics. {V}ol. 2.},
    SERIES = {Cambridge Studies in Advanced Mathematics},
    VOLUME = {62},
 PUBLISHER = {Cambridge University Press},
   ADDRESS = {Cambridge},
      YEAR = {1999},
     PAGES = {xii+581},
      ISBN = {0-521-56069-1; 0-521-78987-7},
   MRCLASS = {05A15 (05-02 05E05 05E10 68R05)},
  MRNUMBER = {MR1676282 (2000k:05026)},
MRREVIEWER = {Ira Gessel},
}

@article {FG,
    AUTHOR = {Fomin, Sergey and Greene, Curtis},
     TITLE = {Noncommutative {S}chur functions and their applications},
      NOTE = {Selected papers in honor of Adriano Garsia (Taormina, 1994)},
   JOURNAL = {Discrete Math.},
  FJOURNAL = {Discrete Mathematics},
    VOLUME = {193},
      YEAR = {1998},
    NUMBER = {1-3},
     PAGES = {179--200},
      ISSN = {0012-365X},
     CODEN = {DSMHA4},
   MRCLASS = {05E05 (20C30)},
  MRNUMBER = {1661368 (2000c:05149)},
MRREVIEWER = {Ernesto Vallejo},
       DOI = {10.1016/S0012-365X(98)00140-X},
       URL = {http://dx.doi.org.proxy.lib.umich.edu/10.1016/S0012-365X(98)00140-X},
}

@article {GesselViennot,
    AUTHOR = {Gessel, Ira and Viennot, G{\'e}rard},
     TITLE = {Binomial determinants, paths, and hook length formulae},
   JOURNAL = {Adv. in Math.},
  FJOURNAL = {Advances in Mathematics},
    VOLUME = {58},
      YEAR = {1985},
    NUMBER = {3},
     PAGES = {300--321},
      ISSN = {0001-8708},
     CODEN = {ADMTA4},
   MRCLASS = {05A15 (05A19)},
  MRNUMBER = {815360 (87e:05008)},
MRREVIEWER = {S. Milne},
       DOI = {10.1016/0001-8708(85)90121-5},
       URL = {http://dx.doi.org/10.1016/0001-8708(85)90121-5},
}

@article {LeeCylSchur,
    AUTHOR = {Lee, Seung Jin},
     TITLE = {Positivity of cylindric skew {S}chur functions},
   JOURNAL = {J. Combin. Theory Ser. A},
  FJOURNAL = {Journal of Combinatorial Theory. Series A},
    VOLUME = {168},
      YEAR = {2019},
     PAGES = {26--49},
      ISSN = {0097-3165},
   MRCLASS = {05E05 (05A05)},
  MRNUMBER = {3960164},
MRREVIEWER = {Markus E. Fulmek},
       DOI = {10.1016/j.jcta.2019.05.009},
       URL = {https://doi-org.offcampus.lib.washington.edu/10.1016/j.jcta.2019.05.009},
}

@article {PostnikovCylSchur,
    AUTHOR = {Postnikov, Alexander},
     TITLE = {Affine approach to quantum {S}chubert calculus},
   JOURNAL = {Duke Math. J.},
  FJOURNAL = {Duke Mathematical Journal},
    VOLUME = {128},
      YEAR = {2005},
    NUMBER = {3},
     PAGES = {473--509},
      ISSN = {0012-7094},
   MRCLASS = {05E05 (14M15 14N15 14N35)},
  MRNUMBER = {2145741},
MRREVIEWER = {Harry Tamvakis},
       DOI = {10.1215/S0012-7094-04-12832-5},
       URL = {https://doi-org.offcampus.lib.washington.edu/10.1215/S0012-7094-04-12832-5},
}

@article {McNamara,
    AUTHOR = {McNamara, Peter},
     TITLE = {Cylindric skew {S}chur functions},
   JOURNAL = {Adv. Math.},
  FJOURNAL = {Advances in Mathematics},
    VOLUME = {205},
      YEAR = {2006},
    NUMBER = {1},
     PAGES = {275--312},
      ISSN = {0001-8708},
   MRCLASS = {05E05 (06A07 14N35)},
  MRNUMBER = {2254313},
MRREVIEWER = {Thomas Fun Yau Lam},
       DOI = {10.1016/j.aim.2005.07.011},
       URL = {https://doi.org/10.1016/j.aim.2005.07.011},
}

@article {GKcyl,
    AUTHOR = {Gessel, Ira M. and Krattenthaler, Christian},
     TITLE = {Cylindric partitions},
   JOURNAL = {Trans. Amer. Math. Soc.},
  FJOURNAL = {Transactions of the American Mathematical Society},
    VOLUME = {349},
      YEAR = {1997},
    NUMBER = {2},
     PAGES = {429--479},
      ISSN = {0002-9947},
   MRCLASS = {05A15 (05A17 05A30 33D20 33D80)},
  MRNUMBER = {1389777},
MRREVIEWER = {Mikl\'{o}s B\'{o}na},
       DOI = {10.1090/S0002-9947-97-01791-1},
       URL = {https://doi.org/10.1090/S0002-9947-97-01791-1},
}

@article{Isaiah,
 author = {Siegl, Isaiah},
 title = {Cylindric {{\(P\)}}-tableaux for {{\((3+1)\)}}-free posets},
 fjournal = {Combinatorial Theory},
 journal = {Comb. Theory},
 issn = {2766-1334},
 volume = {5},
 number = {3},
 pages = {22},
 note = {Id/No 18},
 year = {2025},
 language = {English},
 doi = {10.5070/C65365566},
 zbMATH = {8121954},
 Zbl = {1575.05219}
}

@misc{guaypaquetCohom,
      title={A second proof of the {S}hareshian--{W}achs conjecture, by way of a new {H}opf algebra}, 
      author={Mathieu Guay-Paquet},
      year={2016},
      howpublished = {Preprint, {arXiv}: 1601.05498},
      eprint={1601.05498},
      archivePrefix={arXiv},
      primaryClass={math.CO},
      url={https://arxiv.org/abs/1601.05498}, 
}

@misc{GPchromatic,
	Author = {Guay-Paquet, Mathieu},
	howpublished = {Preprint, {arXiv}:1306.2400},
	Title = {A modular relation for the chromatic symmetric functions of (3+1)-free posets},
	Year = {2013}}

@Article{BrosnanChow,
 Author = {Brosnan, Patrick and Chow, Timothy Y.},
 Title = {Unit interval orders and the dot action on the cohomology of regular semisimple {Hessenberg} varieties},
 FJournal = {Advances in Mathematics},
 Journal = {Adv. Math.},
 ISSN = {0001-8708},
 Volume = {329},
 Pages = {955--1001},
 Year = {2018},
 Language = {English},
 DOI = {10.1016/j.aim.2018.02.020},
 zbMATH = {6863462},
 Zbl = {1410.05222}
}

@article{Gasharov,
	Author = {Gasharov, Vesselin},
	Journal = {Discrete Math.},
	Pages = {193--197},
	Title = {Incomparability graphs of $(3+1)$-free posets are $s$-positive},
	Volume = {157},
	Year = {1996}}

@article{GasharovClaw,
title = {On Stanley's chromatic symmetric function and clawfree graphs},
journal = {Discrete Mathematics},
volume = {205},
number = {1},
pages = {229-234},
year = {1999},
issn = {0012-365X},
doi = {https://doi.org/10.1016/S0012-365X(99)00106-5},
url = {https://www.sciencedirect.com/science/article/pii/S0012365X99001065},
author = {Vesselin Gasharov}
}

@misc{HikitaChromatic,
      title={A proof of the {S}tanley--{S}tembridge conjecture}, 
      author={Tatsuyuki Hikita},
      year={2024},
      month = oct,
      howpublished = {Preprint, {arXiv}:2410.12758}
}

@book{MOAmaj,
 author = {Marshall, Albert W. and Olkin, Ingram and Arnold, Barry C.},
 title = {Inequalities: theory of majorization and its applications},
 edition = {2nd edition},
 fseries = {Springer Series in Statistics},
 series = {Springer Ser. Stat.},
 issn = {0172-7397},
 isbn = {978-0-387-40087-7; 978-0-387-68276-1},
 year = {2011},
 publisher = {New York, NY: Springer},
 language = {English},
 doi = {10.1007/978-0-387-68276-1},
 zbMATH = {5829875},
 Zbl = {1219.26003}
}

@article {Stanleychromatic,
    AUTHOR = {Stanley, Richard P.},
     TITLE = {A symmetric function generalization of the chromatic
              polynomial of a graph},
   JOURNAL = {Adv. Math.},
  FJOURNAL = {Advances in Mathematics},
    VOLUME = {111},
      YEAR = {1995},
    NUMBER = {1},
     PAGES = {166--194},
      ISSN = {0001-8708},
   MRCLASS = {05E05 (05C15)},
  MRNUMBER = {1317387},
MRREVIEWER = {SeungKyung Park},
       DOI = {10.1006/aima.1995.1020},
       URL = {https://doi.org/10.1006/aima.1995.1020},
}

@article{StanleyStembridge,
	Author = {Stanley, Richard P. and Stembridge, John R.},
	Doi = {10.1016/0097-3165(93)90048-D},
	Fjournal = {Journal of Combinatorial Theory. Series A},
	Issn = {0097-3165},
	Journal = {J. Combin. Theory Ser. A},
	Mrclass = {05E15 (20C30)},
	Mrnumber = {1207737},
	Mrreviewer = {A. Kerber},
	Number = {2},
	Pages = {261--279},
	Title = {On immanants of {J}acobi-{T}rudi matrices and permutations with restricted position},
	Url = {https://doi-org.ezp1.lib.umn.edu/10.1016/0097-3165(93)90048-D},
	Volume = {62},
	Year = {1993}}

@article{Sra,
 author = {Sra, Suvrit},
 title = {On inequalities for normalized {Schur} functions},
 fjournal = {European Journal of Combinatorics},
 journal = {Eur. J. Comb.},
 issn = {0195-6698},
 volume = {51},
 pages = {492--494},
 year = {2016},
 language = {English},
 doi = {10.1016/j.ejc.2015.07.005},
 zbMATH = {6487490},
 Zbl = {1321.05286}
}

@article{SWchromatic,
	Author = {Shareshian, John and Wachs, Michelle L.},
	Journal = {Adv. Math.},
	Pages = {497--551},
	Title = {Chromatic quasisymmetric functions},
	Volume = {295},
	Year = {2016}}

@article{CGSsymMeans,
 author = {Cuttler, Allison and Greene, Curtis and Skandera, Mark},
 title = {Inequalities for symmetric means},
 fjournal = {European Journal of Combinatorics},
 journal = {Eur. J. Comb.},
 issn = {0195-6698},
 volume = {32},
 number = {6},
 pages = {745--761},
 year = {2011},
 language = {English},
 doi = {10.1016/j.ejc.2011.01.020},
 zbMATH = {5937082},
 Zbl = {1229.05267}
}

@Article{StanleyChromatic1998,
 Author = {Stanley, Richard P.},
 Title = {Graph colorings and related symmetric functions: {I}deas and applications: {A} description of results, interesting applications, and notable open problems.},
 FJournal = {Discrete Mathematics},
 Journal = {Discrete Math.},
 ISSN = {0012-365X},
 Volume = {193},
 Number = {1-3},
 Pages = {267--286},
 Year = {1998},
 Language = {English},
 DOI = {10.1016/S0012-365X(98)00146-0},
 zbMATH = {1525037},
 Zbl = {1061.05508}
}

\end{document}